\documentclass[12pt]{article}

\usepackage[margin=1in]{geometry}

\usepackage{amsmath, amssymb, amsthm, amssymb}
\usepackage{amsfonts}
\usepackage{graphicx}
\usepackage[frame,ps,matrix,arrow,curve,rotate,color]{xy}
\usepackage{enumerate}
\usepackage{hyperref}
\usepackage{MnSymbol}
\usepackage{xcolor}

\usepackage{blkarray}

\newtheorem{theorem}{Theorem}[section]

\newtheorem{conjecture}[theorem]{Conjecture}
\newtheorem{corollary}[theorem]{Corollary}

\newcounter{claims}[theorem]
\newtheorem{claim}[claims]{Claim}

\theoremstyle{definition}
\newtheorem{definition}[theorem]{Definition}

\newtheorem{question}[theorem]{Question}

\theoremstyle{remark}

\newcommand{\mc}[1]{\mathcal{#1}}

\newcommand{\la}{\langle}
\newcommand{\ra}{\rangle}

\newcommand{\comp}{\mathtt{c}}
\newcommand{\infi}{\mathtt{in}}

\newcommand{\bfSigma}{\boldsymbol{\Sigma}}

\DeclareMathOperator{\Iso}{Iso}
\DeclareMathOperator{\sr}{sr}
\DeclareMathOperator{\SR}{SR}
\DeclareMathOperator{\wf}{wf}
\DeclareMathOperator{\wfc}{wfc}
\DeclareMathOperator{\sh}{sh}

\newcommand{\bfDelta}{\mathbf{\Delta}}
\newcommand{\bfPi}{\mathbf{\Pi}}

\DeclareMathOperator{\Mod}{Mod}

\makeatletter
\newcommand\mathcircled[1]{%
  \mathpalette\@mathcircled{#1}%
}
\newcommand\@mathcircled[2]{%
  \tikz[baseline=(math.base)] \node[draw,circle,inner sep=1pt] (math) {$\m@th#1#2$};%
}
\makeatother

\begin{document}

\title{An Introduction to the Scott Complexity of Countable Structures and a Survey of Recent Results}
\author{Matthew Harrison-Trainor}

\maketitle	

\begin{abstract}
	Every countable structure has a sentence of the infinitary logic $\mc{L}_{\omega_1 \omega}$ which characterizes that structure up to isomorphism among countable structures. Such a sentence is called a Scott sentence, and can be thought of as a description of the structure. The least complexity of a Scott sentence for a structure can be thought of as a measurement of the complexity of describing the structure. We begin with an introduction to the area, with short and simple proofs where possible, followed by a survey of recent advances.
\end{abstract}

\section{Introduction}

A central feature of elementary first-order logic is the compactness theorem, which has as a consequence the Ryll-Nardzewski theorem that any countably-categorical structure is relatively simple: for each $n$, there are only finitely many automorphism orbits of $n$-tuples. So most structures are not countably categorical, which means that no matter which elementary first-order sentences we write down, we cannot completely characterize the structure up to isomorphism.

If we strengthen our language, Scott \cite{Scott65} showed that we can characterize every countable structure up to isomorphism. The required strengthening of the language allows countably infinite conjunctions and disjunctions, so that for example we can express that an abelian group is torsion by saying that for each element $x$, either $x = 0$, or $x + x = 0$, or $x + x + x = 0$, and so on. More formally, we use the logic $\mc{L}_{\omega_1 \omega}$ which is defined as follows:

\begin{definition}
	Let $\tau$ be a vocabulary. The formulas of the logic $\mc{L}_{\omega_1 \omega}(\tau)$ are built inductively as follows:
	\begin{itemize}
		\item terms and atomic formulas are defined as in elementary first-order logic;
		\item if $\varphi$ is a formula, so is $\neg \varphi$;
		\item if $\bar{x}$ is a finite tuple of variables and $X$ is a countable set of formulas such that for each $\varphi \in X$, the free variables of $\varphi$ are contained within $\bar{x}$, then $\bigdoublewedge_{\varphi \in X} \varphi$ and $\bigdoublevee_{\varphi \in X} \varphi$ are formulas;
		\item if $\varphi$ is a formula, then so are $\forall x \varphi$ and $\exists x \varphi$.
	\end{itemize}
\end{definition}

\noindent $\mc{L}_{\omega_1 \omega}$ does not satisfy compactness, but it does satisfy downward (though not upward) Lowenheim-Skolem for countable fragments. The book by Marker \cite{MarkerBook} is an excellent reference on infinitary model theory.

Now that we have strengthened our language, it is possible to show that each structure is characterized up to isomorphism among countable structures by a sentence of this logic \cite{Scott65}. We call this sentence a \textit{Scott sentence}.

\begin{definition}
	Let $\mc{A}$ be a countable structure. A \textit{Scott sentence} for $\mc{A}$ is a sentence $\varphi$ of $\mc{L}_{\omega_1 \omega}$ such that, up to isomorphism, $\mc{A}$ is the only countably model of $\varphi$.
\end{definition}

\noindent The standard proof that every countable structure has a Scott sentence passes through the back-and-forth relations, showing that they must stabilize at some countable ordinal, and then extracting from this a Scott sentence. This \textit{Scott analysis} of a structure has played an important role in the study of Vaught's conjecture, e.g.\ in Morley's theorem \cite{Morley70} that the number of non-isomorphic countable models of a theory is either at most $\aleph_1$ or is exactly $2^{\aleph_0}$. There are excellent references for these results on Vaught's conjecture, e.g.\ in Marker's book \cite{MarkerBook}, and so we will only touch on them briefly.

Fixing a countable structure $\mc{A}$, there is a some countable ordinal at which the back-and-forth relations stabilize. This introduces a way of assigning an ordinal to $\mc{A}$, called its \textit{Scott rank}. There are several slightly different variations of what it means to stabilize, leading to several different definitions of Scott rank. The Scott rank can be thought of equivalently as measuring both the complexity of the automorphism orbits of the structure and the complexity of describing the structure.

Another way of measuring the complexity of $\mc{A}$ is the least complexity of a Scott sentence for $\mc{A}$. To measure the complexity of $\mc{L}_{\omega_1 \omega}$ formulas, we use the natural hierarchy of $\Sigma_\alpha$ and $\Pi_\alpha$ sentences in normal form.
\begin{itemize}
	\item An $\mc{L}_{\omega_1 \omega}$ formula is both $\Sigma_0$ and $\Pi_0$ if it is quantifier free and does not contain any infinite disjunction or conjunction. 
	\item An $\mc{L}_{\omega_1 \omega}$ formula is $\Sigma_\alpha$ if it is a countable disjunction of formulas of the form $\exists x \phi$ where each $\phi$ is $\Pi_\beta$ for some $\beta < \alpha$.
	\item An $\mc{L}_{\omega_1 \omega}$ formula is $\Pi_\alpha$ if it is a countable disjunction of formulas of the form $\forall x \phi$ where each $\phi$ is $\Sigma_\beta$ for some $\beta < \alpha$.
\end{itemize}
This approach led Montalb\'an \cite{Montalban15} to define a new notion of Scott rank as the least ordinal $\alpha$ for which $\mc{A}$ has a $\Pi_{\alpha+1}$ Scott sentence. He proved that this is robust in the sense that there are a number of different equivalent characterizations, including the complexity of defining automorphism orbits, and the complexity of computing isomorphisms between different copies of $\mc{A}$.

An even finer measurement, the \textit{Scott complexity} of $\mc{A}$, is the least complexity of a Scott sentence for $\mc{A}$. Alvir, Greenberg, Harrison-Trainor, and Turetsky \cite{AlvirGreenbergHTTuretsky} argued by means of the Wadge degrees that the only (reasonable) possible complexities are $\Sigma_\alpha$, $\Pi_\alpha$, and $\mathrm{d-}\Sigma_\alpha$.

\medskip{}

This article will begin with a more detailed look at the above background material. Following this, we will touch on more specialized areas. We begin with effective considerations and structures of high Scott rank. Here we give new elementary proofs, avoiding Barwise compactness and $\Sigma^1_1$ bounding, which we think are simpler than many of those in the literature. Next, we look at results on the Scott complexity of particular structures of interest in mathematics, e.g.\ free groups. Finally, we discuss theories, including a brief summary of connections with Vaught's conjecture.

\medskip{}

We also note that there is a work applying the Scott analysis to separable metric spaces, which we will not cover in this survey. See \cite{MelnikovNies,Doucha,YaacovDouchaNiesTsankov,ChanChen,Chan}; note that \cite{Doucha} has an erratum \cite{DouchaErratum}.

\section{Scott Sentences and Scott Rank}

Our first step is to see why every structure has a Scott sentence. The standard way to see this is to analyse the back-and-forth relations. There are a number of different ways to define the back-and-forth relations, but from a computability-theoretic perspective the most natural is the following due to its connections with $\Sigma_\alpha$ and $\Pi_\alpha$ sentences.

\begin{definition}\label{def:bfasym}
	The \textit{standard asymmetric back-and-forth relations} $\leq_\alpha$ on $\mc{A}$, for $\alpha < \omega_1$, are defined by:
	\begin{enumerate}
		\item $\bar{a} \leq_0 \bar{b}$ if $\bar{a}$ and $\bar{b}$ satisfy the same quantifier-free formulas from among the first $|\bar{a}|$-many formulas.
		\item For $\alpha > 0$, $\bar{a} \leq_\alpha \bar{b}$ if for each $\beta < \alpha$ and $\bar{d}$ there is $\bar{c}$ such that $\bar{b} \bar{d} \leq_\beta \bar{a} \bar{c}$.
	\end{enumerate}
	We define $\bar{a} \equiv_\alpha \bar{b}$ if $\bar{a} \leq_\alpha \bar{b}$ and $\bar{b} \leq_\alpha \bar{a}$. We also define $\bar{a} \equiv_\infty \bar{b}$ if $\bar{a} \equiv_\alpha \bar{b}$ for all $\alpha < \omega_1$.
\end{definition}

\noindent It is not hard to prove by induction that for $\alpha \geq 1$, $\bar{a} \leq_\alpha \bar{b}$ if and only if every $\Sigma_\alpha$ sentence true of $\bar{b}$ is true of $\bar{a}$. Similarly, $\bar{a} \equiv_\infty\bar{b}$ if and only if $\bar{a}$ and $\bar{b}$ satisfy the same $\mc{L}_{\omega_1 \omega}$ formulas.

For each pair of tuples $\bar{a}$ and $\bar{b}$ such that $\bar{a} \not\equiv_\infty \bar{b}$, choose a formula $\varphi_{\bar{a};\bar{b}}(\bar{x})$ such that $\mc{A} \models \varphi_{\bar{a};\bar{b}}(\bar{a})$ but $\mc{A} \models \neg \varphi_{\bar{a};\bar{b}}(\bar{b})$. Then define $\psi_{\bar{a}}(\bar{x}) := \bigdoublewedge_{\bar{b} \not\equiv_\infty \bar{a}} \varphi_{\bar{a};\bar{b}}(\bar{x})$. So we have that $\mc{A} \models \psi_{\bar{a}}(\bar{b})$ if and only if $\bar{a} \equiv_\infty \bar{b}$. We may also add as a conjunct to $\psi_{\bar{a}}(\bar{x})$ all of the atomic formulas true of $\bar{a}$.

Now define:
\begin{align*}
	\rho_\varnothing &:= \forall x \bigdoublevee_{b \in \mc{A}}  \psi_b(x) \\
	\rho_{\bar{a}} &:= \forall \bar{u} \left ( \psi_{\bar{a}}(\bar{u}) \longrightarrow \forall x \bigdoublevee_{b \in \mc{A}}  \psi_{\bar{a}b}(\bar{u},x) \right)\\
	\sigma_\varnothing &:= \bigdoublewedge_{b \in \mc{A}}  \exists x \psi_b(x) \\
	\sigma_{\bar{a}} &:=  \forall \bar{u} \left (\psi_{\bar{a}}(\bar{u}) \longrightarrow  \bigdoublewedge_{b \in \mc{A}} \exists x  \psi_{\bar{a}b}(\bar{u},x) \right)
\end{align*}
and
\[ \theta := \bigdoublewedge_{\bar{a} \in \mc{A}} \rho_{\bar{a}} \wedge \sigma_{\bar{a}}.\]
It is easy to argue by a standard back-and-forth argument that $\theta$ is a Scott sentence for $\mc{A}$. We call $\theta$ the \textit{canonical Scott sentence} for $\mc{A}$.

\begin{theorem}[Scott \cite{Scott65}]\label{thm:scott}
	Every countable structure has a Scott sentence.
\end{theorem}
\begin{proof}
	It is clear that $\mc{A} \models \theta$. Suppose that $\mc{B}$ is countable and $\mc{B} \models \theta$. Suppose that we have a finite partial map $\mc{A} \to \mc{B}$ mapping $\bar{a}$ to $\bar{b}$ with $\psi_{\bar{a}}(\bar{b})$. Then for such $\bar{a}$ and $\bar{b}$:
	\begin{itemize}
		\item $\bar{a}$ and $\bar{b}$ have the same atomic type;
		\item given $a' \in \mc{A}$, since $\mc{B} \models \sigma_{\bar{a}} \wedge \psi_{\bar{a}}(\bar{b})$, there is $b' \in \mc{B}$ with $\mc{B} \models \psi_{\bar{a}a'}(\bar{b}b')$; 
		\item given $b' \in \mc{B}$, since $\mc{B} \models \rho_{\bar{a}} \wedge \psi_{\bar{a}}(\bar{b})$, there is $a' \in \mc{A}$ with $\mc{B} \models \psi_{\bar{a}a'}(\bar{b}b')$; 
	\end{itemize}
	So by a back-and-forth argument we can extend our finite partial isomorphism to an isomorphism $\mc{A} \to \mc{B}$. Thus $\theta$ is a Scott sentence for $\mc{A}$.
\end{proof}

As a corollary, by adding constants to the language, we get:

\begin{corollary}
	Let $\mc{A}$ be a countable structure and $\bar{a},\bar{b} \in \mc{A}$. Then there is an automorphism of $\mc{A}$ taking $\bar{a}$ to $\bar{b}$ if and only if $\bar{a} \equiv_\infty \bar{b}$.
\end{corollary}

There are only countably many tuples $\bar{a}$ in $\mc{A}$, so there is a countable ordinal $\alpha$ such, that for $\bar{a},\bar{b} \in \mc{A}$, if $\bar{a} \not\equiv_\infty \bar{b}$ then $\bar{a} \nleq_\alpha \bar{b}$. We may choose the formulas $\varphi_{\bar{a},\bar{b}}$ to be $\Pi_\alpha$, so that $\psi_{\bar{a}}$ are also $\Pi_\alpha$. Then the formula $\Theta$ is $\Pi_{\alpha + 2}$. So there is a natural notion of rank that we can assign to $\mc{A}$ based on the least such ordinal $\alpha$. Ash and Knight \cite{AshKnight00} prefer the following ranks:

\begin{definition}\label{def:asym-sr}
	For $\bar{a} \in \mc{A}$, define $r(\bar{a})$ as the least ordinal $\alpha$ such that for all $\bar{b}$, if $\bar{a} \leq_\alpha \bar{b}$ then $\bar{a} \equiv_\infty \bar{b}$. Then define:
	\begin{itemize}
		\item $r(\mc{A}) = \sup \{r(\bar{a}) : \bar{a} \in \mc{A}\}$;
		\item $R(\mc{A}) = \sup \{r(\bar{a})+1 : \bar{a} \in \mc{A}\}$;
	\end{itemize}
\end{definition}
\noindent It is possible for a two structures $\mc{A}$ and $\mc{B}$ to have $r(\mc{A}) = r(\mc{B}) = \lambda$ a limit ordinal, but $R(\mc{A}) = \lambda$ and $R(\mc{B}) = \lambda + 1$. 

Scott's original proof of Theorem \ref{thm:scott} used symmetric back-and-forth relations $\sim_\alpha$ instead of the asymmetric ones defined above. In the definition below we extend only by single elements $c$ and $d$ rather than by tuples to show another possible variation.

\begin{definition}\label{def:bfsym}
	The \textit{standard symmetric back-and-forth relations} $\sim_\alpha$ on $\mc{A}$, for $\alpha < \omega_1$, are defined by:
	\begin{enumerate}
		\item $\bar{a} \sim_0 \bar{b}$ if $\bar{a}$ and $\bar{b}$ satisfy the same quantifier-free formulas.
		\item For $\alpha > 0$, $\bar{a} \sim_\alpha \bar{b}$ if for each $\beta < \alpha$ and $d$ there is $c$ such that $\bar{a} c \sim_\beta \bar{b} d$, and for all $c$ there is $d$ such that $\bar{a} c \sim_\beta \bar{b} d$.
	\end{enumerate}
\end{definition}
\noindent Again, one can argue that there is a countable ordinal $\alpha$ such that if $\bar{a} \sim_\alpha \bar{b}$, then $\bar{a} \sim_\beta \bar{b}$ for all $\beta$. We can define a rank based on these back-and-forth relations.
\begin{definition}
	For $\bar{a} \in \mc{A}$, define $\sr(\bar{a})$ as the least ordinal $\alpha$ such that for all $\bar{b}$, if $\bar{a} \sim_\alpha \bar{b}$ then $\bar{a} \equiv_\infty \bar{b}$. Then define:
	\begin{itemize}
		\item $\sr(\mc{A}) = \sup \{\sr(\bar{a}) : \bar{a} \in \mc{A}\}$;
		\item $\SR(\mc{A}) = \sup \{\sr(\bar{a})+1 : \bar{a} \in \mc{A}\}$;
	\end{itemize}
\end{definition}
\noindent This is what Marker \cite{MarkerBook} uses as the definition of Scott rank. It is from this connection with Scott's original proof of Theorem \ref{thm:scott} that the name \textit{Scott rank} is given to these notions of rank, though it is unclear which rank should hold the title. In general, we prefer the ranks $R$ and $\SR$ over the ranks $r$ and $\sr$, as they contain more information. These ranks $R$ and $\SR$ agree at multiples of $\omega^2$.

There are two other possible variations on the back-and-forth relations: defining the symmetric relations $\sim_\alpha$ to extend by tuples (which is a common definition) and defining the asymmetric relations $\leq_\alpha$ to extend by single elements (which is rarely seen). Each of these definitions also gives rise to a notion of rank.

\medskip

Another common notion of rank, used, e.g., by Sacks \cite{Sacks07} and in Marker's book on model theory \cite{MarkerMTBook}, is as follows.
\begin{definition}\label{def:scott-analysis}
	Let $\mc{A}$ be a countable structure. We define fragments $\mc{L}_{\alpha}^{\mc{A}}$ of $\mc{L}_{\omega_1 \omega}$ and $\mc{L}_{\alpha}^{\mc{A}}$-theories $T_{\alpha}^{\mc{A}}$ as follows:
	\begin{enumerate}
		\item $\mc{L}_0^{\mc{A}}$ is finitary first-order logic;
		\item $T_\alpha^{\mc{A}}$ is the complete theory of $\mc{A}$ in $\mc{L}_{\alpha}^{\mc{A}}$;
		\item $\mc{L}_{\lambda}^{\mc{A}}$ is the union of $\mc{L}_{\alpha}^{\mc{A}}$ for $\alpha < \lambda$ a limit ordinal; and
		\item $\mc{L}_{\alpha+1}^{\mc{A}}$ is the least fragment of $\mc{L}_{\omega_1 \omega}$ containing $\mc{L}_{\alpha}^{\mc{A}}$ and also containing, for each non-principal $n$-type $p$ of $T_\alpha^{\mc{A}}$, the conjunction $\bigwedge p$ of all the formulas in $p$.
	\end{enumerate}
\end{definition}
\noindent Essentially, $T_{\alpha+1}^{\mc{A}}$ gives a name to all of the non-principal types of $T_\alpha^{\mc{A}}$, though this may then introduce new non-principal types. However, one can show that this process stabilizes and for some $\alpha$, all of the types of $T_{\alpha}^{\mc{A}}$ realized in $\mc{A}$ are principal. One can then define the Scott rank of $\mc{A}$ to be the least $\alpha$ such that $\mc{A}$ is the atomic model of $T_\alpha^{\mc{A}}$. This rank agrees with the previous ranks at multiples of $\omega^2$. This notion of rank can be useful when one wants to apply ideas from finitary model theory.

\medskip

Gao \cite{Gao} also introduces a notion of Scott rank based on back-and-forth games, which is similar in spirit to the ranks defined using back-and-forth relations.

\medskip

A final notion of Scott rank was introduced by Montalb\'an \cite{Montalban15} based on the following list of equivalent properties:

\begin{theorem}[Montalb\'an \cite{Montalban15}]\label{thm:montalban}
	Let $\mc{A}$ be a countable structure, and $\alpha$ a countable ordinal. The following are equivalent:
	\begin{enumerate}
		\item $\mc{A}$ has a $\Pi_{\alpha+1}$ Scott sentence.
		\item Every automorphism orbit in $\mc{A}$ is $\Sigma_\alpha$-definable without parameters.
		\item $\mc{A}$ is uniformly (boldface) $\mathbf{\Delta}^0_\alpha$-categorical.
	\end{enumerate}
\end{theorem}

\noindent Montalb\'an proposed the following definition of Scott rank:

\begin{definition}\label{def:Scott-rank-mon}
	The \textit{Scott rank} of a countable structure $\mc{A}$ is the least ordinal $\alpha$ such that $\mc{A}$ has a $\Pi_{\alpha + 1}$ Scott sentence.
\end{definition}

\noindent Montalb\'an argued that this is the correct notion of Scott rank because it is robust in the sense that it has many equivalent characterizations via Theorem \ref{thm:montalban}. It is this definition that we use for most of the rest of this article, except when otherwise noted.

There is a clear connection between this notion of Scott rank proposed by Montalb\'an and the rank $R$ of Definition \ref{def:asym-sr} defined using the asymmetric back-and-forth relations $\leq_\alpha$. Given $\bar{a} \in \mc{A}$, the formula $\psi_{\bar{a}}$ defines the automorphism orbit of $\bar{a}$; this is because if $\bar{a} \equiv_\infty \bar{b}$, then the structures $(\mc{A},\bar{a})$ and $(\mc{A},\bar{b})$, with $\bar{a}$ and $\bar{b}$ named as constants, are isomorphic. With $\alpha = r(\bar{a})$, the automorphism orbit of $\bar{a}$ is defined by a $\Pi_{\alpha}$ formula. Thus the Scott rank as defined by Montalb\'an and the rank $R$ differ by at most one.

\section{Scott Complexity}

There is some arbitrariness in choosing to define Scott rank using $\Pi$ formulas rather than $\Sigma$ formulas. Indeed, if we defined the Scott rank of a structure to be the least ordinal $\alpha$ such that the structure has a $\Sigma_{\alpha+2}$ Scott sentence, then we would have the desirable property that naming constants does not change Scott rank: $\mc{A}$ has the same Scott rank as $(\mc{A},\bar{c})$ for any tuple $\bar{c} \in \mc{A}$. This follows from the following fact:
\begin{theorem}[Montalb\'an]\label{thm:ex-ss}
	Let $\mc{A}$ be a countable structure. Then $\mc{A}$ has a $\Sigma_{\alpha+1}$ Scott sentence if and only if for some $\bar{c} \in \mc{A}$, $(\mc{A},\bar{c})$ has a $\Pi_{\alpha}$ Scott sentence.
\end{theorem}
\noindent The proof of this fact is less obvious than it seems; indeed it was stated as an obvious fact without proof in \cite{Montalban15} but a proof was first provided (for $\Sigma_3$ sentences) in \cite{Montalban17}. A full proof appears in \cite{AlvirGreenbergHTTuretsky}.

Moreover, defining Scott rank in this way is still robust; combining Theorems \ref{thm:montalban} and \ref{thm:ex-ss}, we get:
\begin{theorem}\label{thm:montalban2}
	Let $\mc{A}$ be a countable structure, and $\alpha$ a countable ordinal. The following are equivalent:
	\begin{enumerate}
		\item $\mc{A}$ has a $\Sigma_{\alpha+2}$ Scott sentence.
		\item There are parameters $\bar{c}$ such that every automorphism orbit in $\mc{A}$ is $\Sigma_\alpha$-definable over $\bar{c}$.
		\item $\mc{A}$ is (boldface) $\mathbf{\Delta}^0_\alpha$-categorical.
	\end{enumerate}
\end{theorem}

\noindent In \cite{MontalbanBook2}, Montalb\'an refers to the Scott rank of Definition \ref{def:Scott-rank-mon} as \textit{parameterless Scott rank} and the Scott rank defined using $\Sigma$ formulas as \textit{parametrized Scott rank}, and uses the latter as the definition.

Rather than choosing between these two possibilities, one can give up on assigning a single ordinal rank, and rather say that the \textit{Scott complexity} (or \textit{Scott sentence complexity}) of a structure is the least complexity of a Scott sentence for that structure, e.g., if a structure has a $\Pi_3$ Scott sentence but no $\Sigma_3$ Scott sentence, then its Scott complexity is $\Pi_3$. (This is not yet a formal definition, as we have not specified the possible complexities of sentences.)

Now these complexities are not totally ordered; $\Pi_\alpha$ and $\Sigma_\alpha$ are incomparable. So a natural question is: Can a structure have both a $\Sigma_{\alpha+1}$ Scott sentence and a $\Pi_{\alpha+1}$ Scott sentence, but no $\Sigma_\alpha$ or $\Pi_\alpha$ Scott sentence? The answer turns out to be yes---for example, the group $\mathbb{Z}$ has a $\Sigma_3$ and a $\Pi_3$ Scott sentence, but no $\Sigma_2$ or $\Pi_2$ Scott sentence---but the following theorem of A. Miller says that the structure must then have a Scott sentence which is a conjunction of a $\Sigma_\alpha$ sentence and a $\Pi_\alpha$ sentence. We call such a sentence a $\mathrm{d-}\Sigma_\alpha$ sentence, where d stands for \textit{difference}, because it is the difference of two $\Sigma_\alpha$ sentences.
\begin{theorem}[A. Miller \cite{AMiller}]\label{thm:akm}
	Let $\mc{A}$ be a countable structure. Then if $\mc{A}$ has both a $\Sigma_\alpha$ and a $\Pi_\alpha$ Scott sentence, it has a $\mathrm{d-}\Sigma_\beta$ Scott sentence for some $\beta < \alpha$.
\end{theorem}
	
\noindent So for example the group $\mathbb{Z}$ has a $\mathrm{d-}\Sigma_2$ sentence but no $\Sigma_2$ or $\Pi_2$ sentence. Thus when we say that the Scott complexity of a structure is the least complexity of a Scott sentence for that structure, we have to include not just $\Sigma_\alpha$ and $\Pi_\alpha$ as complexities, but $\mathrm{d-}\Sigma_2$. But among these complexities, there is always a single least complexity of a Scott sentence.

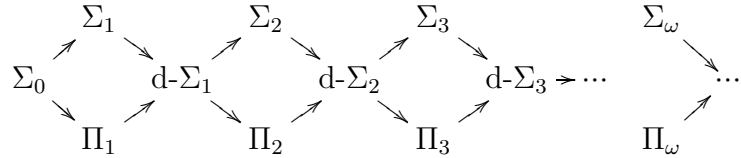
\[
\xymatrix@=7pt{
	& \Sigma_1 \ar[dr]                                      &                                                      &\Sigma_2\ar[dr]&& \Sigma_3 \ar[dr] & && \Sigma_\omega\ar[dr] \\
	\Sigma_0 \ar[ur]\ar[dr] &&\text{d-}\Sigma_1 \ar[ur]\ar[dr]&&\text{d-}\Sigma_2 \ar[ur]\ar[dr] & & \text{d-}\Sigma_3 \ar[r] & \cdots&&\cdots\\
	& \Pi_1 \ar[ur]                                            &                                                     &\Pi_2\ar[ur]&& \Pi_3    \ar[ur] & && \Pi_\omega\ar[ur]
}
\]

Are there further possible least complexities of Scott sentences? For all of the other complexity classes of sentences we can think of, any structure with a Scott sentence of that complexity also has a simpler Scott sentence. Let us give two examples. First, if a structure has a Scott sentence which is a negation of a $\mathrm{d-}\Sigma_\alpha$ sentence, that is, of the form $\varphi \vee \psi$ where $\varphi$ is $\Sigma_\alpha$ and $\psi$ is $\Pi_\alpha$, then one of either $\varphi$ or $\psi$ is a Scott sentence on its own, and so the structure has either a $\Sigma_\alpha$ or $\Pi_\alpha$ Scott sentence. Second, there are complexities $n\mathrm{-}\Sigma_\alpha$ corresponding to the difference hierarchy. Suppose that a structure has a Scott sentence which is $3\mathrm{-}\Sigma_\alpha$, that is, of the form $(\varphi \wedge \neg \psi) \vee \theta$ where these sentences are all $\Sigma_\alpha$. Then either the $\mathrm{d-}\Sigma_\alpha$ sentence $\varphi \wedge \neg \psi$ is a Scott sentence, or the $\Sigma_\alpha$ sentence $\psi$ is a Scott sentence. A similar argument works for the rest of the difference hierarchy. Empirically through arguments such as these, and by checking examples, it appears that $\Sigma_\alpha$, $\Pi_\alpha$, and $\mathrm{d-}\Sigma_\alpha$ are the only possible Scott complexities. Thus we define:

\begin{definition}\label{def:scott-sentence-complexity}
	The \emph{Scott sentence complexity} of a structure $\mc{A}$ is the least complexity, from among $\Sigma_\alpha$, $\Pi_\alpha$, and $\mathrm{d-}\Sigma_\alpha$, of a Scott sentence for $\mc{A}$.
\end{definition} 

It is desirable to have some kind of more rigorous argument that these should be the possible complexities, and that it is not just that we cannot think of any other possibilities. For this, we use a perspective that began in \cite{AMiller} of looking at the Borel complexity of the set of copies of a structure. Given a language $\mc{L}$, there is a Polish space $\Mod(\mc{L})$ of all $\mc{L}$-structures with domain $\omega$. Given a structure $\mc{A} \in \Mod(\mc{L})$, let $\Iso(\mc{A})$ be the set of all isomorphic copies of $\mc{A}$ in $\Mod(\mc{L})$. There is a correspondence between Scott sentences of $\mc{A}$ and the complexity of $\Iso(\mc{A})$ via the Lopez-Escobar theorem, and its strengthenings by Vaught and D. Miller.
\begin{theorem}[Lopez-Escobar \cite{LE}, Vaught \cite{Vaught}, D. Miller \cite{DMiller}]
	Let $\mathbb{K}$ be a subclass of $\Mod(\mc{L})$ which is closed under isomorphism. Then $\mathbb{K}$ is $\bfSigma^0_\alpha$ if and only if $\mathbb{K}$ is axiomatized by an infinitary $\Sigma_\alpha$ sentence. The same is true for $\bfPi_\alpha$ and $\Pi_\alpha$ sentences, and for the difference hierarchy.
\end{theorem}
\noindent Since each structure has a Scott sentence, $\Iso(\mc{A})$ is always Borel.

Intuitively, the Lopez-Escobar result should extend to any natural and robust notion of complexity for sentences. So in place of looking at the least complexity of a Scott sentence for $\mc{A}$, we look at the complexity of $\Iso(\mc{A})$. This gives us access to additional structure via Wadge reducibility.
\begin{definition}[Wadge]
	Let $A$ and $B$ be subsets of $\Mod(\mc{L})$. We say that $A$ is \emph{Wadge reducible} to $B$, and write $A \leq_W B$, if there is a continuous function $f$ with $A = f^{-1}[B]$, i.e.
	\[x \in A \Longleftrightarrow f(x) \in B .\]
\end{definition}

\noindent The equivalence classes under this pre-order are called the Wadge degrees; we write $[A]_W$ for the Wadge degree of $A$. The Wadge hierarchy is the set of Wadge degrees under continuous reductions.

With enough determinacy, the Wadge hierarchy is very well-behaved; it is well-founded and almost totally ordered (in the sense that any anti-chain has size at most two). 
\begin{theorem}[Martin and Monk, AD]
	The Wadge order is well-founded.
\end{theorem}
\begin{theorem}[Wadge's Lemma, AD, \cite{Wadge}]
	Given $A,B \subseteq \omega^\omega$, either $A \leq_W B$ or $B \leq_W \omega^\omega - A$.
\end{theorem}

\noindent Since determinacy for Borel sets is provable in ZFC, this theorem holds in ZFC for such sets. Since $\Iso(\mc{A})$ is Borel, we will not need to use any determinacy assumptions.

In general, for each of the pointclasses $\Gamma$ from among $\bfSigma^0_\alpha$, $\bfPi^0_\alpha$, $\bfDelta^0_\alpha$, $\mathrm{d-}\bfSigma^0_\alpha$, and other pointclasses arising from the Borel or difference hierarchies, if $A$ is Wadge-reducible to a set in $\Gamma$, then $A$ itself is in $\Gamma$; and moreover, there is a $\Gamma$-complete set. We denote by $\Gamma$ the Wadge degree of a $\Gamma$-complete set. So, for example, $\bfSigma^0_1$ is the Wadge degree of open, but not clopen, sets.

\medskip{}

Now we define:
\begin{definition}\label{def:scott-complexity}
	The \emph{Scott complexity} of a structure $\mc{A}$ is the Wadge degree of $\Iso(\mc{A})$.
\end{definition} 

\noindent Of course we want to show that this is the same as Definition \ref{def:scott-sentence-complexity}. (We have made a slight differentiation between the two by calling one Scott sentence complexity and the other Scott complexity; they are technically different types of objects, the former being a complexity class of sentences and the latter being a Wadge degree, though by the Lopez-Escobar theorem the two are in correspondence.) The possible Scott complexities are exactly what we expect given the discussion above: $\bfPi^0_\alpha$, $\bfSigma^0_\alpha$, and $\mathrm{d-}\bfSigma^0_\alpha$. There are certain values of $\alpha$ for which some of these are not possible, e.g.\ $\Sigma_2$ is not the Scott complexity of any structure. A complete list is as follows:

\begin{theorem}[A. Miller \cite{AMiller}, Alvir, Greenberg, Harrison-Trainor, and Turetsky \cite{AlvirGreenbergHTTuretsky}]
	The possible Scott complexities of countable structures $\mc{A}$ are:
	\begin{enumerate}
		\item $\bfPi^0_\alpha$ for $\alpha \geq 1$,
		\item $\bfSigma^0_\alpha$ for $\alpha \geq 3$ a successor ordinal,
		\item $\mathrm{d-}\bfSigma^0_\alpha$ for $\alpha \geq 1$ a successor ordinal.
	\end{enumerate}
	There is a countable structure with each of these Wadge degrees.
\end{theorem}
\noindent A. Miller constructed a number of examples of structures with various Scott complexities, and also ruled out some possibilities, such as $\bfSigma^0_2$ (for relational languages). He left open the case of constructing a structure of Scott complexity $\bfSigma^0_{\alpha+1}$ for $\alpha$ a limit ordinal. Alvir, Greenberg, Harrison-Trainor, and Turetsky found such examples, and showed that these are the only possibilities.

As a consequence of this theorem, \textit{Scott sentence complexity} as in Definition \ref{def:scott-sentence-complexity} and \textit{Scott complexity} as in Definition \ref{def:scott-complexity} are really the same thing, under the natural correspondence between Wadge degrees of $\Iso(\mc{A})$ and Scott sentences for $\mc{A}$, that is, identifying $\Sigma_\alpha$ and $\bfSigma^0_\alpha$, etc.

\section{Structures of High Scott Rank}

We now turn to an effective analysis of Scott complexity. We say that an ordinal $\alpha$ is $X$-computable if there is an $X$-computable linear order isomorphic to $\alpha$. As there are only countably many $X$-computable ordinals, there must be a least non-$X$-computable ordinal, which we call $\omega_1^{X}$ (or $\omega_1^{CK}$ when $X$ is computable). Since any initial segment of an $X$-computable well-founded linear order has a $X$-computable presentation, the $X$-computable ordinals are closed downwards; the $X$-computable ordinals are exactly the ordinals below $\omega_1^X$.

We will give simpler and more elementary proofs than are usual in the literature, avoiding the use of $\Sigma^1_1$ bounding, ordinal notations, or the Barwise compactness theorem. (See, e.g., \cite{AshKnight00} for proofs of some of these results using the Barwise compactness theorem.) Instead, our main tool will be the fact that if an $X$-computable tree is well-founded, then its tree rank is an $X$-computable ordinal. Recall that the tree rank is defined as follows:
\begin{itemize}
	\item $rk(x) = 0$ if $x$ is a leaf.
	\item $rk(x)$ is otherwise the least ordinal (or possibly $\infty$) greater than the ranks of the children of $x$.
\end{itemize}
The rank of a tree is the rank of its root node.

\begin{theorem}
	Let $T$ be an $X$-computable well-founded tree. Then the tree rank of $T$ is an $X$-computable ordinal.
\end{theorem}
\begin{proof}[Proof sketch]
	Given a tree $T \subseteq \omega^{< \omega}$ we will define computably in $T$ a linear order, called the Kleene-Brouwer order. If $T$ is well-founded then the Kleene-Brouwer order will be as well; and in this case, the order type of the Kleene-Brouwer order will be an ordinal greater than the tree rank of $T$. The Kleene-Brouwer order is defined on the nodes of $T$ by $s \leq_{KB} t$ if and only if
	\begin{itemize}
		\item $t \preceq s$ ($s$ extends $t$) or 
		\item $s(n) < t(n)$ and $t \upharpoonright n = s \upharpoonright n$ ($s$ is to the left of $t$).
	\end{itemize}
	If an $X$-computable tree $T$ is well-founded with tree rank $\alpha$, the Kleene-Brouwer order of $T$ can be computed from $T$ and hence is an $X$-computable presentation of an ordinal $\beta \geq \alpha$. Thus $\alpha$ itself is $X$-computable. We leave the details to the reader; they are not difficult.
\end{proof}

\subsection{Bounds on Scott complexity}

There are countably many computable structures, and hence by a simple counting argument there must be some ordinal $\alpha < \omega_1$ which is an upper bound on the Scott rank of a computable structure. Our goal in this section is to compute this bound, as well as to prove some other results along the way.

For this section, it is helpful to extend the back-and-forth relations to tuples from different structures.
\begin{definition}
	Let $\mc{A}$ and $\mc{B}$ be structures. The \textit{standard asymmetric back-and-forth relations} $\leq_\alpha$, for $\alpha < \omega_1$, are defined for $\bar{a} \in \mc{A}$ and $\bar{b} \in \mc{B}$ of the same length by:
	\begin{enumerate}
		\item $(\mc{A},\bar{a}) \leq_0 (\mc{B},\bar{b})$ if $\bar{a}$ and $\bar{b}$ satisfy the same quantifier-free formulas from among the first $|\bar{a}|$-many formulas.
		\item For $\alpha > 0$, $(\mc{A},\bar{a}) \leq_\alpha (\mc{B},\bar{b})$ if for each $\beta < \alpha$ and $\bar{d} \in \mc{B}$ there is $\bar{c} \in \mc{A}$ such that $(\mc{B},\bar{b} \bar{d}) \leq_\beta (\mc{A},\bar{a} \bar{c})$.
	\end{enumerate}
	We define $\equiv_\infty$ etc.\ as before.
\end{definition}

To begin, suppose that we have a computable structure $\mc{A}$. 

\begin{theorem}[Nadel \cite{Nadel}]\label{thm:comp-bf}
	Let $\mc{A}$ and $\mc{B}$ be $X$-computable structures. If $\mc{A} \equiv_\alpha \mc{B}$ for every $\alpha < \omega_1^X$, then $\mc{A} \cong \mc{B}$.
\end{theorem}
\begin{proof}
	A finite partial isomorphism maps an element $\bar{a}$ of $\mc{A}$ to a tuple $\bar{b}$ of $\mc{B}$ such that $(\mc{A},\bar{a}) \equiv_0 (\mc{B},\bar{b})$: $\bar{a}$ and $\bar{b}$ satisfy the same atomic formulas from among the first $|\bar{a}|$-many. We can form a tree of finite partial isomorphisms; the root node is the empty partial isomorphism, and the children of a node are the finite partial isomorphisms extending that node. 
	
	Let $T$ be the subtree of the tree of finite partial isomorphisms such that at the $n$th level, all of the finite partial isomorphism contain the first $n$ elements of $\mc{A}$ in their domain, and the first $n$ elements of $\mc{B}$ in their range. So $T$ is $X$-computable. A path through $T$ corresponds to a (total) isomorphism $\mc{A} \to \mc{B}$.
	
	\begin{claim}
		Consider a node $\sigma$ of $T$ corresponding to a finite partial isomorphism $\bar{a} \mapsto \bar{b}$. If $(\mc{A},\bar{a}) \equiv_{\alpha \cdot 4} (\mc{B},\bar{b})$, then the tree rank of $\sigma$ is at least $\alpha$.
	\end{claim}
	\begin{proof}
		We argue by induction on $\alpha$. We must show that for each $\beta < \alpha$, there is a child of $\sigma$ of tree rank at least $\beta$. Suppose without loss of generality that $\sigma$ is at the $n$th level of the tree.
		
		First, let $d$ be the $n+1$st element of $\mc{B}$. (If $d$ is already contained in $\bar{b}$, then do nothing here; we assume that $d$ is not in $\bar{b}$.) Since $(\mc{A},\bar{a}) \equiv_{\alpha \cdot 4} (\mc{B},\bar{b})$, there is $c \in \mc{A}$ such that $(\mc{A},\bar{a}c) \equiv_{\beta \cdot 4 + 2} (\mc{B},\bar{b}d)$. Now let $c' \in \mc{A}$ be the $n+1$st element of $\mc{A}$. (Again, if $c'$ is already contained in $\bar{a}c$, then do nothing here; we assume that $c'$ is not in $\bar{a}c$.) Let $d' \in \mc{B}$ be such that $(\mc{A},\bar{a}cc') \equiv_{\beta \cdot 4} (\mc{B},\bar{b}dd')$. Then $\bar{a}cc' \mapsto \bar{b}dd'$ is a finite partial isomorphism which is a child of $\sigma$, and by the induction hypothesis, it has tree rank at least $\beta$.
	\end{proof}

	Now $(\mc{A},\varnothing) \equiv_{\omega_1^X} (\mc{B},\varnothing)$ and so the tree rank of $T$ is at least $\omega_1^X$. Since $T$ is $X$-computable, $T$ must have a path, which is an isomorphism $\mc{A} \to \mc{B}$.
\end{proof}

We can easily apply this to tuples in a single structure.

\begin{corollary}[Nadel \cite{Nadel}]\label{cor:comp-bf}
	Let $\mc{A}$ be an $X$-computable structures and $\bar{a},\bar{b} \in \mc{A}$. If $\bar{a} \equiv_\alpha \bar{b}$ for every $\alpha < \omega_1^X$, then there is an automorphism of $\mc{A}$ taking $\bar{a}$ to $\bar{b}$.
\end{corollary}

We also get an upper bound on the Scott complexity of a computable structure.

\begin{corollary}[Nadel \cite{Nadel}]\label{cor:upper-bound}
	An $X$-computable structure has a $\Pi_{\omega_1^{CK}+2}$ Scott sentence, and hence has Scott rank at most $\omega_1^{CK}+1$.
\end{corollary}

This is true for all of the different definitions of Scott rank.

\begin{proof}
	By the previous corollary, if $\bar{a} \leq_{\omega_1^X} \bar{b}$, then $\bar{a}\equiv_\infty \bar{b}$. Then as in the discussion preceding Definition \ref{def:asym-sr}, we can choose the formulas $\varphi_{\bar{a},\bar{b}}$ in the definition of the canonical Scott sentence to be $\Pi_{\omega_1^X}$, so that the resulting Scott sentence is $\Pi_{\omega_1^X + 2}$.
\end{proof}

Next we will consider the computability of the Scott sentences of a computable structure. We give a slightly informal definition of the computable formulas of $\mc{L}_{\omega_1 \omega}$. For the fully formal definition, see \cite{AshKnight00}. 

\begin{definition}{\ }
	\begin{itemize}
		\item A computable $\Sigma_0$ or $\Pi_0$ formula is just a finitary quantifier-free formula.
		\item A computable $\Sigma_\alpha$ formula is a disjunction of formulas $\exists \bar{x} \psi(\bar{x})$ where the formulas $\psi(\bar{x})$ are computable $\Pi_\beta$ for some $\beta < \alpha$, and the there is a c.e.\ set of codes for the formulas over which the disjunction is being taken.
		\item A computable $\Pi_\alpha$ formula is a conjunction of formulas $\exists \bar{x} \psi(\bar{x})$ where the formulas $\psi(\bar{x})$ are computable $\Sigma_\beta$ for some $\beta < \alpha$, and the there is a c.e.\ set of codes for the formulas over which the conjunction is being taken.
	\end{itemize}
\end{definition}

Let $\mc{A}$ be an $X$-computable structure. We will define, recursively for $\alpha < \omega_1^X$ and tuples $\bar{a} \in \mc{A}$, formulas
\begin{itemize}
	\item $\varphi_{\bar{a},\alpha}$ such that $\mc{B} \models \varphi_{\bar{a},\alpha}(\bar{b})$ if and only if $(\mc{A},\bar{a}) \leq_\alpha (\mc{B},\bar{b})$; and
	\item $\psi_{\bar{a},\alpha}$ such that $\mc{B} \models \psi_{\bar{a},\alpha}(\bar{b})$ if and only if $(\mc{B},\bar{b}) \leq_\alpha (\mc{A},\bar{a})$.
\end{itemize}
We put $\varphi_{\bar{a},0}(\bar{x}) = \psi_{\bar{a},0}(\bar{x})$ to be the conjunction of the first $|\bar{x}|$-many atomic and negated atomic formulas about $\bar{x}$. For $\alpha > 0$, we put 
\[ \varphi_{\bar{a},\alpha}(\bar{x}) = \bigwedge_{\beta < \alpha} \; \forall \bar{y} \; \bigvee_{\bar{c} \in \mc{A}} \psi_{\bar{a}\bar{c},\beta}(\bar{x},\bar{y})  \]
and
\[ \psi_{\bar{a},\alpha}(\bar{x}) = \bigwedge_{\beta < \alpha} \; \bigwedge_{\bar{c} \in \mc{A}} \exists \bar{y} \; \varphi_{\bar{a}\bar{c},\beta}(\bar{x},\bar{y}).\]
Both $\varphi_{\bar{a},\alpha}$ and $\psi_{\bar{a},\alpha}$ are computable $\Pi_{\alpha \cdot 2}$ formulas.

Suppose that $\alpha$ is an $X$-computable ordinal such that for all $\bar{a},\bar{b} \in \mc{A}$, if $\bar{a} \leq_\alpha \bar{b}$ then $\bar{a} \equiv_\infty \bar{b}$. Then $\mc{A} \models \psi_{\bar{a},\alpha}(\bar{b})$ if and only if $\bar{a} \equiv_\infty \bar{b}$. In the definition of the canonical Scott sentence, we may use $\psi_{\bar{a},\alpha}$ as the formula $\psi_{\bar{a}}$. We get that the resulting Scott sentence is a computable $\Pi_{\alpha \cdot 2 + 2}$ sentence. (Note that this argument only works if the ordinal $\alpha$ is $X$-computable.) Thus we get:

\begin{theorem}[Nadel \cite{Nadel}]
	Let $\mc{A}$ be an $X$-computable structure.  Then $\mc{A}$ has Scott rank $< \omega_1^X$ if and only if $\mc{A}$ has an $X$-computable Scott sentence.
\end{theorem}

\noindent An interesting line of investigation is to study the difference between the least-complexity Scott sentence and the least-complexity computable Scott sentence for a structure. Not much is known, though for example Ho constructed a subgroup of $\mathbb{Q}$ with a $\mathrm{d-}\Sigma_2$ Scott sentence but no computable $\mathrm{d-}\Sigma_2$ Scott sentence \cite{Ho}. (The structure has computable $\Sigma_3$ and $\Pi_3$ structures, so that Theorem \ref{thm:akm} is not true for computable Scott sentences.)

Finally, we will prove effective versions of Theorem \ref{thm:comp-bf} and Corollary \ref{cor:comp-bf}.

\begin{theorem}[Nadel \cite{Nadel}]\label{thm:comp-formulas-cat}
	Let $\mc{A}$ and $\mc{B}$ be $X$-computable structures. If $\mc{A}$ and $\mc{B}$ satisfy the same $X$-computable infinitary formulas, then they are isomorphic.
\end{theorem}
\begin{proof}
	For each $\alpha < \omega_1^X$, we have that $\mc{B} \models \varphi_{\varnothing,\alpha}$ where this is the $X$-computable formula defined above. This implies that $\mc{A} \leq_\alpha \mc{B}$ for each $\alpha$. Thus $\mc{A}$ and $\mc{B}$ are isomorphic by Theorem \ref{thm:comp-bf}.
\end{proof}

\begin{corollary}[Nadel \cite{Nadel}]
	Let $\mc{A}$ be an $X$-computable structure. If $\bar{a}$ and $\bar{b} \in \mc{A}$ satisfy the same $X$-computable formulas, then there is an automorphism of $\mc{A}$ taking $\bar{a}$ to $\bar{b}$.
\end{corollary}

\subsection{Construction of structures of high Scott complexity}

Corollary \ref{cor:upper-bound} is the best possible bound on the Scott complexity of an $X$-computable structure. Harrison \cite{Harrison68} constructed a computable linear order of order type $\omega_1^{CK} \cdot (1 + \mathbb{Q})$, which has Scott complexity $\Pi_{\omega_1^{CK}+2}$. It is a \textit{pseudo-well-ordering} in the sense that it has no hyperarithmetic descending sequence. More generally, we can construct the Harrison linear order $\omega_1^{X} \cdot (1 + \mathbb{Q})$ relative to any set $X$, which has Scott complexity $\Pi_{\omega_1^{X}+2}$. Often, the Barwise compactness theorem is used to construct the Harrison linear order, but we will give a more natural construction found in \cite{Kleene55} and which is explained in Lemma III.2.1 of \cite{SacksBook}. Our account comes from \cite{Chan17}. We will give the general overview of the argument giving most but not all of the details.

The relation ``$Y$ is not hyperarithmetic in $X$'' is $\Sigma^1_1(X)$ and so there is an $X$-computable tree $T$ whose paths are pairs $\langle Y,f \rangle$ where $f$ witnesses that $Y$ is not hyperarithmetic in $X$. $T$ is $X$-computable uniformly in $X$ and has no $X$-hyperarithmetic path. Take the Kleene-Brouwer order of $T$. We get an $X$-computable linear order $L$ with no $X$-hyperarithmetic descending sequence, as given a descending sequence in $L$, with one more jump we can compute a path through $T$. Then we argue that $L$ has order type $\omega_1^X \cdot (1 + \mathbb{Q}) + \alpha$ for some $\alpha < \omega_1^X$. This takes several steps:
\begin{itemize}
	\item Let $\gamma$ be the order type of the well-founded initial segment of $L$. Write $L = L_1 + L_2$ with $L_1$ the well-founded initial segment of $L$. If $\gamma$ were a computable ordinal, then we could hyperarithmetically split $L$ as $L_1 + L_2$, and $L_2$ has no least element, giving a hyperarithmetic descending sequence in $L$. Also, we cannot have $\gamma > \omega_1^X$ as then we could fix some element $u$ such that $\{ v \in L : v < u\}$ is a computable linear ordering of order type $\omega_1^X$. Thus $\gamma = \omega_1^X$.
	\item Similarly, given any $u \in L$, either $\{ v \in L : v \geq u\}$ is well-founded or the well-founded initial segment of $\{ v \in L : v \geq u\}$ has order type $\omega_1^X$. Thus $L \cong \omega_1^X \cdot (1+L^*) + \alpha$ for some linear order $L^*$ and ordinal $\alpha$.
	\item $\alpha$ must be computable, as if $u$ is the initial element of $\alpha$, $\{v \in L : v \geq u\}$ is a computable linear order.
	\item We claim that $L^*$ is dense and without endpoints. If $L^*$ had a left endpoint, then $\omega_1^X + \omega_1^X$ would be an initial segment of $L$, which we have already argued is not the case. If it had a right endpoint, then $\omega_1^X + \alpha$ would be a final segment of $L$, which cannot happen because $\omega_1^X + \alpha$ is not $X$-computable. And if $L^*$ was not dense, then we could write $L \cong \cdots + \omega_1^X + \omega_1^X + \cdots$, and this cannot happen as $\omega_1^X$ is not $X$-computable.
\end{itemize}
Finally, we use a trick to remove the $\alpha$: we replace $L$ by $L \cdot \omega$ which has order type $\omega_1^X \cdot (1 + \mathbb{Q})$.

Now we have to argue that the Scott complexity of $\omega_1^X \cdot (1 + \mathbb{Q})$ is $\Pi_{\omega_1^{X}+2}$. We must show that it has no $\Sigma_{\omega_1^X+2}$ Scott sentence. Using Theorem \ref{thm:montalban2}, it suffices to show that there is no tuple $\bar{c}$ such that the automorphism orbits are definable by $\Sigma_{\alpha}$ formulas, $\alpha < \omega_1^X$. (Note that if an orbit is definable by a $\Sigma_{\omega_1^X}$ formula, then it is definable by a $\Sigma_{\alpha}$ formula for some $\alpha < \omega_1^X$.) The parameters $\bar{c}$ do not help  us define the automorphism orbits because, given $\bar{c} = (c_1,\ldots,c_n)$ with $c_1 < c_2 < \cdots < c_n$, between some two of the parameters is a linear order of order type $\omega_1^X \cdot (1 + \mathbb{Q}) + \beta$ for some $\beta < \omega_1^X$. So let us ignore the parameters. (See, for example, Section 15.3.3 of \cite{AshKnight00} for similar arguments.) Let $u$ be an element of the non-well-founded part of the Harrison linear order, and suppose to the contrary that $\varphi(\bar{x})$ is a $\Sigma_\alpha$ formula defining the automorphism type of $u$, with $\alpha < \omega_1^X$. Then (using the formula $\varphi_{u,\alpha}$ from the previous section) there is an $X$-computable $\Pi_{\alpha \cdot 2}$ formula defining the automorphism orbit of $u$. So we may assume, after increasing $\alpha$, that $\varphi$ is $X$-computable. So the automorphism orbit $\{v \in L : L \models \varphi(v)\}$ of $u$ is $X$-hyperarithmetic, and has no least element. Then we can find a hyperarithmetic descending sequence in $L$.

Thus we obtain:

\begin{theorem}
	Given $X \in 2^\omega$, there is an $X$-computable linear order of Scott rank $\omega_1^{X}+1$ and Scott complexity $\Pi_{\omega_1^{X}+2}$. Moreover, there is a uniform operator $\Phi$ such that $\Phi(X)$ has order type $\omega_1^X \cdot (1 + \mathbb{Q})$, so that $\omega_1^X = \omega_1^Y$ if and only if $\Phi^X \cong \Phi^Y$.
\end{theorem}

In general, we say that a structure $\mc{A}$ has \textit{high Scott rank} / \textit{high Scott complexity} if satisfies one of the following equivalent conditions:
\begin{itemize}
	\item it does not have an $X$-computable Scott sentence;
	\item its Scott rank is not $X$-computable, i.e.\ at least $\omega_1^{\mc{A}}$;
	\item it has Scott complexity at least $\Pi_{\omega_1^{\mc{A}}}$.
\end{itemize}
All of the different possible definitions of Scott rank agree on whether a structure has high Scott rank, and if it does have high Scott rank, they agree on whether the Scott rank is $\omega_1^X$ or $\omega_1^X + 1$.

In terms of Scott complexity, there are five possible Scott complexities for a structure of high Scott complexity:
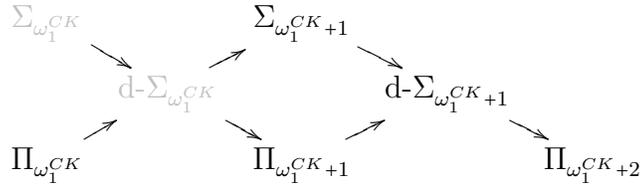
\[
\xymatrix@=7pt{
	& {\color{lightgray} \Sigma_{\omega_1^{CK}} \ar@[lightgray][dr]}                                      &                                                      &\Sigma_{\omega_1^{CK}+1}\ar[dr] \\
	&&{\color{lightgray} \text{d-}\Sigma_{\omega_1^{CK}}} \ar[ur]\ar[dr]&&\text{d-}\Sigma_{\omega_1^{CK}+1} \ar[dr] \\
	& \Pi_{\omega_1^{CK}} \ar[ur]                                            &                                                     &\Pi_{\omega_1^{CK}+1}\ar[ur]&& \Pi_{\omega_1^{CK}+2} 
}
\]
Structures of Scott complexity $\Pi_{\omega_1^{CK}}$ and $\Pi_{\omega_1^{CK}}+1$ have Scott rank $\omega_1^{CK}$, and structures of the other three Scott complexities have Scott rank $\omega_1^{CK}+1$.

While a computable structure with Scott rank $\omega_1^{CK}+1$ was known in the 1960's, it was not until the 2000's that a computable structure of Scott rank $\omega_1^{CK}$ was constructed. \cite{Makkai} gave a complicated construction of a $\Delta^0_2$ structure with Scott rank $\omega_1^{CK}$, and later Knight and Millar gave a simpler construction while also obtaining a computable structure \cite{KnightMillar}. The structures they build have Scott complexity Scott complexity $\Pi_{\omega_1^{CK}}$.

\begin{theorem}[Knight and Millar \cite{KnightMillar}]\label{thm:knight-millar}
	There is a computable structure of Scott rank $\omega_1^{CK}$ and Scott complexity $\Pi_{\omega_1^{CK}}$. Moreover, given $X$, there is an $X$-computable structure of Scott rank $\omega_1^{X}$ and Scott complexity $\Pi_{\omega_1^{X}}$.
\end{theorem}
\begin{proof}[Proof sketch]
	The construction of Knight and Millar uses the Barwise compactness theorem to build a computable thin homogeneous tree, where these terms are defined below. We will give a proof along the same lines, building a computable thin homogeneous tree, but we will build this tree in a more constructive and avoid using the Barwise compactness theorem. We will give a computable construction but the construction relativizes to any $X$.
	
	\begin{definition}
		A tree $T$ is \textit{thin} if there is a computable ordinal bound on the ordinal tree ranks at each level of the tree.
	\end{definition}

	\noindent Note that there can be nodes on the tree of rank $\infty$; the bound is just on the nodes with ordinal rank.

	\begin{definition}
		A tree $T$ is \textit{homogeneous} if:
		\begin{itemize}
			\item Whenever $x$ has a successor of rank $\alpha$, it has infinitely many successors of rank $\alpha$.
			\item If some element at level $n$ has a successor of rank $\alpha$, every element at level $n$ with rank $> \alpha$ has a successor of rank $\alpha$.
		\end{itemize}
	\end{definition}

	\noindent Such a tree is called homogeneous because any two nodes at the same level of the tree and with the same rank are automorphic, and indeed the isomorphism type of the tree is determined by which tree ranks appear at each level. One can show that a computable thin homogeneous tree has a $\Pi_{\omega_1^{CK}}$ Scott sentence; one writes down the conjunctions of the sentences, one for each level, which say which tree ranks show up at that level. Each of these sentences is $\Pi_{\omega_1^{CK}}$ because there is a bound on the tree ranks at that level. Though the tree ranks at each level are bounded, if they are unbounded below $\omega_1^{CK}$ overall, the tree has Scott complexity $\Pi_{\omega_1^{CK}}$. See \cite{KnightMillar} for details.
	
	\medskip
	
	So we want to construct a thin homogeneous tree with unbounded ranks. To construct such a tree, first fix a computable presentation $\mc{H}$ of the Harrison linear order. By the trick of replacing each element of $\mc{H}$ by a copy of $\omega$, we may assume that we can compute successors, predecessors, and limits in $\mc{H}$.	We construct for each $n = 1,2,\ldots$ a set $S_n \subseteq \mc{H}$ such that:
	\begin{enumerate}
		\item $S_1 \subseteq S_2 \subseteq S_3 \subseteq \cdots$,
		\item $\bigcup_i S_i = \mc{H}$,
		\item $S_1$ is cofinal in $\mc{H}$,
		\item each $S_n$ is bounded above inside the well-founded part of $\mc{H}$,
		\item for each $a \in S_n$ which is a limit element of $\mc{H}$, $\{b \in S_{n+1} : b < a\}$ is unbounded below $a$,
		\item for each $a \in S_n$ which is a successor element of $\mc{H}$, the predecessor of $a$ is in $S_{n+1}$.
	\end{enumerate}
	To construct $S_1$, consider the elements of $\mc{H}$ in order, and put each element which is larger than all of the previous elements into $S_1$. For example, if the elements of $\mc{H}_1$ are $h_1,h_2,h_3,\ldots$, and the first nine elements are ordered as
	\[ h_9 < h_2 < h_5 < h_1 < h_6 < h_7 < h_3 < h_4 < h_8 \]
	then we put $h_1$ in $S_1$, omit $h_2$ because it is smaller than $h_2$, put $h_3$ and $h_4$ in $S_1$, then omit $h_5$, $h_6$, and $h_7$ because they are lower than $h_4$, put $h_8$ in $S_1$, and omit $h_9$. So we get $S_1 = \{h_1,h_3,h_4,h_8,\ldots\}$.  Note that the order type of $S_1$ is $\omega$. Only finitely many of the elements of $S_1$ are in the well-founded part of $\mc{H}$; this is for (4).
	
	Given $S_1,\ldots,S_n$, we construct $S_{n+1}$ as follows. Each element of $S_n$ will be an element of $S_{n+1}$, satisfying (1). For each element $y$ of $S_{n} - S_{n-1}$ which is a limit element of $\mc{H}$, we can find $x \in S_{n}$ such that there are no element $z$ of $S_{n+1}$ with $x < z < y$. Then, in the open interval $\{ z \in \mc{H} : x < z < y\}$ we can do the same construction as we did for $S_1$, listing out the elements in this interval and putting each element which is greater than the previous elements into $S_{n+1}$. Thus we satisfy (5). For each successor element of $S_n$, we put its predecessor into $S_{n+1}$, satisfying (6). To satisfy (2), we also have to add to $S_{n+1}$ the first element of $\mc{H}$ which does not yet appear in $S_n$. (4) is easy to verify, and (3) is just about $S_1$.
	
	\medskip
	
	Let $T$ a subtree of the tree of all descending sequences in $\mc{H}$, with the nodes at level $n$ being from $S_n$; i.e., the path $\bar{h} = \la h_1,\ldots,h_n \ra$ has $h_1 \in S_1, h_2 \in S_2,\ldots,h_n \in S_n$. Given a node corresponding to the sequence $\bar{h} = \la h_1,\ldots,h_n \ra$, it is easy to argue using (5) and (6) that the tree rank of $\bar{h}$ is the order type of $\{x \in \mc{H} : x < h_n\}$ if $h_n$ is in the well-founded part of $\mc{H}$, and $\infty$ otherwise. By choice of the $S_n$---namely (4)---$T$ is a thin tree. To make it homogeneous, we just need to duplicate each node infinitely many times. By (3), the root node---which is the empty sequence---has tree rank $\infty$, but by (2), there is no bound on the tree ranks at all levels.
\end{proof}

The construction of the Harrison linear order was uniform in the sense that given $X$, one could effectively construct the Harrison linear order relative to $X$. Moreover, if $\omega_1^X = \omega_1^Y$, then the Harrison linear order relative to $X$ is the same as the Harrison linear order relative to $Y$. One can view this as saying that the construction produces a natural/canonical structure, \textit{the} Harrison linear order. On the other hand, the construction just given of an $X$-computable structure of Scott rank $\omega_1^X$ is not uniform in this way. The structure one gets depends on the sets $S_1,S_2,\ldots$, and these in turn depend on the particular presentation of the Harrison linear order $\mc{H}$.

Chan \cite{Chan17} asked whether there is some other uniform construction of a computable structure of Scott rank $\omega_1^{CK}$. Becker \cite{Becker17} and Chan, Harrison-Trainor, and Marks (unpublished) independently showed that there is no uniform construction.
\begin{theorem}[Becker \cite{Becker17}; Chan, Harrison-Trainor, and Marks]
	There is no Borel operator $\Phi$ such that:
	\begin{itemize}
		\item for all $X,Y \in 2^\omega$, $\Phi(X) \cong \Phi(Y) \Longleftrightarrow \omega_1^X = \omega_1^Y$
		\item for all $X \in 2^\omega$, $\Phi(X)$ is a computable structure of Scott complexity strictly less than $\Pi_{\omega_1^X+2}$.
	\end{itemize}
\end{theorem}
\begin{proof}
	The proof we give is due to an anonymous referee. Suppose to the contrary that there was such a $\Phi$. Since $\Phi$ is Borel, for sufficiently large $\beta$, the inverse image of a $\bfSigma^0_\beta$ set under $\Phi$ is $\bfSigma^0_\beta$. Choose $X$ such that $\omega_1^X > \beta$. Let $\mc{A} = \Phi(X)$. If $\Iso(\mc{A})$ is $\bfSigma^0_{\omega_1^X + 2}$, then so is $\Phi^{-1}[\Iso(\mc{A})]$. But by Theorem 2.3 of \cite{Marker88}, the set $\Phi^{-1}[\Iso(\mc{A})] = \{Y : \omega_1^Y = \omega_1^X\}$ is not $\bfSigma^0_{\omega_1^X + 2}$.
\end{proof}

\noindent Becker \cite{Becker20} later proved a number of strengthenings of this result.

\bigskip{}

Theorem \ref{thm:comp-formulas-cat} says that two computable structures that satisfy the same computable infinitary formulas are isomorphic. Given a computable structure $\mc{A}$, the computable infinitary theory of $\mc{A}$ is the collection of all the computable $\mc{L}_{\omega_1 \omega}$ sentences true of $\mc{A}$. The structure of Scott rank $\omega_1^{CK}$ constructed by Knight and Millar (Theorem \ref{thm:knight-millar}) is the only model, computable or not, of its computable infinitary theory; we say that its computable infinitary theory is countably categorical. On the other hand the Harrison linear order, which has Scott rank $\omega_1^{CK}+1$, is not the only model of its computable infinitary theory. Millar and Sacks \cite{MillarSacks08} asked whether there is a computable structure of Scott rank $\omega_1^{CK}$ such that the computable infinitary theory is not countably categorical. This is equivalent to finding a structure of Scott complexity $\Pi_{\omega_1^{CK}+1}$. (In one direction, if the computable infinitary theory of a computable structure is countably categorical, then the structure has a $\Pi_{\omega_1^{CK}}$ Scott sentence, namely the conjunction of all the computable sentences true of it. The converse is slightly more difficult; see \cite{AlvirGreenbergHTTuretsky}.)

Millar and Sacks \cite{MillarSacks08} produced such a structure $\mc{A}$ which is not hyperarithmetic, but which does have $\omega_1^{\mc{A}} = \omega_1^{CK}$. Freer \cite{Freer08} proved the analogous result higher up: for any admissible $\alpha$, there is a structure $\mc{A}$ with $\omega_1^{\mc{A}} = \alpha$ and $\mc{A}$ has Scott complexity $\Pi_{\omega_1^{\mc{A}}+1}$. Finally, Harrison-Trainor, Igusa, and Knight \cite{HTIgusaKnight18} constructed a computable example.

\begin{theorem}[Harrison-Trainor, Igusa, and Knight \cite{HTIgusaKnight18}]\label{thm:comp-inf}
	There is a computable structure with Scott complexity $\Pi_{\omega_1^{CK}+1}$.
\end{theorem}

There are two remaining complexities, $\Sigma_{\omega_1^{CK}+1}$ and $\mathrm{d-}\Pi_{\omega_1^{CK}+1}$. These are exactly the structures with the property that they have Scott rank $\omega_1^{CK}+1$, but after naming a constant, they have Scott rank $\omega_1^{CK}$. This follows from Theorem \ref{thm:montalban2} and the following corresponding result for $\mathrm{d-}\Sigma_\alpha$ sentences:
\begin{theorem}[Alvir, Greenberg, Harrison-Trainor, and Turetsky \cite{AlvirGreenbergHTTuretsky}]
	Let $\mc{A}$ be a countable structure. Then $\mc{A}$ has a $\mathrm{d-}\Sigma_{\alpha}$ Scott sentence if and only if for some $\bar{c} \in \mc{A}$, $(\mc{A},\bar{c})$ has a $\Pi_{\alpha}$ Scott sentence and the automorphism orbit of $\bar{c}$ in $\mc{A}$ is $\Sigma_{\alpha}$-definable.
\end{theorem}

\noindent Turetsky asked whether such structures exist. Alvir, Greenberg, Harrison-Trainor, and Turetsky \cite{AlvirGreenbergHTTuretsky} showed that they do. Thus there are computable structures of all five possible high Scott complexities.

\begin{theorem}[Alvir, Greenberg, Harrison-Trainor, and Turetsky \cite{AlvirGreenbergHTTuretsky}]\label{thm:other-possibilities}
	There are computable structures of all possible high Scott complexities: $\Pi_{\omega_1^{CK}}$,  $\Pi_{\omega_1^{CK}+1}$,
	$\Sigma_{\omega_1^{CK}+1}$, $\mathrm{d-}\Pi_{\omega_1^{CK}+1}$, and $\Pi_{\omega_1^{CK}+2}$.
\end{theorem}

\subsection{Computably approximable structures}

Theorem \ref{thm:comp-formulas-cat} says that two computable structures that satisfy the same computable infinitary formulas are isomorphic. Thus the computable infinitary theory of a computable structure characterizes the structure among computable structures. (If a computable structure does not have a countably categorical computably infinitary theory, for example the Harrison linear order or the structure of Theorem \ref{thm:comp-inf}, then the other models of its computable infinitary theory are necessarily non-computable.) If we restrict our attention to computable structures only, is there a single sentence characterizing every computable structure? We call such a sentence a \textit{pseudo-Scott sentence}.
\begin{definition}
	Let $\mc{A}$ be a computable structure. A pseudo-Scott sentence for $\mc{A}$ is a sentence whose computable models are just the computable copies of $\mc{A}$.
\end{definition}
\noindent Calvert and Knight \cite{CalvertKnight06} asked whether every computable structure has a computable infinitary pseudo-Scott sentence. Such a structure would have to be of high Scott rank, as structures of computable Scott rank have computable Scott sentences, and any Scott sentence is a pseudo-Scott sentence.
\begin{question}[Calvert and Knight \cite{CalvertKnight06}]
	Is there a computable structure of noncomputable Scott
	rank with a computable infinitary pseudo-Scott sentence?
\end{question}
This question is still open. A structure without a computable Scott sentence is said to be weakly computably approximable.
\begin{definition}
	A computable structure $\mc{A}$ of non-computable rank is \textit{weakly computably approximable} if every computable infinitary sentence $\varphi$ true in $\mc{A}$ is also true in some computable $\mc{B} \ncong \mc{A}$.
\end{definition}
The Harrison linear order has the property that any computable sentence true of it is also true of some other structure which is not just computable, but which has computable Scott rank. Moreover, the Harrison order has the following stronger property:
\begin{definition}
	A computable structure $\mc{A}$ of high Scott rank is \textit{strongly computably approximable} if for all $\Sigma^1_1$ sets $S$, there is a uniformly computable sequence $(\mc{C}_n)$ such that if $n \in S$, then $\mc{C}_n \cong \mc{A}$, and if $n \notin S$, then $\mc{C}_n$ has computable Scott rank.
\end{definition}
\noindent One can see that the Harrison order is strongly computably approximable by a simple modification of the construction given earlier.

Calvert, Knight, and Millar also showed that there is also such a structure of Scott rank $\omega_1^{CK}$. This can also be seen using our construction for Theorem \ref{thm:knight-millar} combined with the fact that the Harrison order is strongly computably approximable.

\begin{theorem}[Calvert, Knight, and Millar \cite{CalvertKnightMillar06}]
	There is a computable tree $T$ of Scott rank $\omega_1^{CK}$ that is strongly computably approximable.
\end{theorem}
\noindent This theorem is very useful for other constructions; for example it was used in the proofs of Theorems \ref{thm:comp-inf} and \ref{thm:other-possibilities}.

Not every computable structure of high Scott rank is strongly computable approximable; indeed, there is a single computable sentence all of whose models have high Scott rank.

\begin{theorem}[Harrison-Trainor \cite{HT18}]\label{thm:comp}
	For $\alpha = \omega_1^{CK}$ or $\alpha = \omega_1^{CK} + 1$ : There is a computable model $\mc{A}$ of Scott rank $\alpha$ and a $\Pi^{\comp}_2$ sentence $\psi$ such that $\mc{A} \models \psi$, and whenever $\mc{B}$ is any structure and $\mc{B} \models \psi$, $\mc{B}$ has Scott rank $\alpha$.
\end{theorem}

Another application of the existence of strongly computably approximable structures is a computation of the index set complexity of being of high Scott rank. 
\begin{theorem}[Calvert, Fokina, Goncharov, Knight, Kudinov, Morozov, and Puzarenko \cite{CFGKKMP}]
	The index set
	\[ \{ i : \text{$\mc{A}_i$ has high Scott rank}\} \]
	of computable structures with high Scott rank is $\Pi^1_1$ $m$-complete.
\end{theorem}
\noindent They also computed the complexity of the index sets of computable structures of Scott rank $\omega_1^{CK}$ and $\omega_1^{CK}+1$; these are complete for differences of $\Pi^1_1$ sets and differences of $\Sigma^1_1$ sets respectively.

\subsection{Computable categoricity and high Scott complexity}

A computable structure $\mc{A}$ is said to be \textit{computably categorical} if whenever $\mc{B} \cong \mc{A}$ is a computable copy of $\mc{A}$, there is a computable isomorphism from $\mc{A}$ to $\mc{B}$. For a long time, an important problem was to characterize the structures which are computably categorical. It turned out that the more well-behaved notion was that of \textit{relative computable categoricity}: a computable structure $\mc{A}$ is relatively computably categorical if for every copy $\mc{B}$ of $\mc{A}$, not necessarily computable, there is an isomorphism from $\mc{A}$ to $\mc{B}$ computable from $\mc{B}$. Relative computable categoricity is characterised by the following theorem:
\begin{theorem}[Ash, Knight, Mannasse, and Slaman \cite{AshKnightManasseSlaman}; Chisholm \cite{Chisholm}]
	A computable structure $\mc{A}$ is relatively computably categorical if and only if it has a computable $\Sigma_1$ \textit{Scott family}: there is a tuple $\bar{c}$ of constants and a c.e.\ collection $\Phi$ of computable $\Sigma_1$ formulas such that
	\begin{enumerate}
		\item each tuple $\bar{a}$ in $\mc{A}$ satisfies $\varphi(\bar{a},\bar{c})$ for some formula $\varphi \in \Phi$; and
		\item whenever two tuples $\bar{a}$ and $\bar{b}$ in $\mc{A}$ satisfy the same formula $\varphi(\cdot,\bar{c})$ from $\Phi$, there is an automorphism of $\mc{A}$ taking $\bar{a}$ to $\bar{b}$.
	\end{enumerate}
\end{theorem}
\noindent It is not hard to see that if $\mc{A}$ has a computable $\Sigma_1$ Scott family, then it is relatively computably categorical; the other direction is a forcing argument. So in some sense if a structure is relatively computably categorical, then there is a good structural reason for this.

As for (plain) computable categoricity, in certain classes there are good structural ways to determine whether a structure is computably categorical---e.g., a linear order is computably categorical if and only if it has finitely many adjacencies \cite{GoncharovDzgoev80}---but there is no good characterization in general. Goncharov \cite{Goncharov77} showed that there are structures that are computably categorical but not relatively computably categorical. Later, Downey, Kach, Lempp,
Lewis, Montalb\'an, and Turetsky \cite{DKLLMT} showed that the set of (indices for) computably categorical structures is $\Pi^1_1$ $m$-complete; this can be interpreted as meaning that there is no good characterization of computable categoricity that is simpler than the naive definition. As part of this proof, they constructed for each computable ordinal $\alpha < \omega_1^{CK}$ a computably categorical structure $\mc{A}$ which is not \textit{relatively $\Delta^0_\alpha$-categorical}, i.e., it is not the case that for every $\mc{B} \cong \mc{A}$ there is an isomorphism from $\mc{A}$ to $\mc{B}$ which is $\Delta^0_\alpha$ in $\mc{B}$.

This left a gap: Is every computably categorical structure $\Delta^0_\alpha$-categorical for some computable ordinal $\alpha < \omega_1^{CK}$? Or even, is every computably categorical structure $\Delta^1_1$-categorical? (A structure is relatively $\Delta^1_1$-categorical if for every $\mc{B} \cong \mc{A}$ there is an isomorphism from $\mc{A}$ to $\mc{B}$ which is $\Delta^1_1$ in $\mc{B}$.) If $\mc{A}$ is a computable structure of computable Scott rank $\alpha < \omega_1^{CK}$, then it is relatively $\Delta^0_{\alpha \cdot 2}$-categorical; this follows from a similar argument as Theorem \ref{thm:comp-formulas-cat}. In the other direction, one can argue that a computable structure which is relatively $\Delta^1_1$-categorical must have computable Scott rank. (See Proposition 1 of \cite{Turetsky}.) So a computable structure is relatively $\Delta^1_1$-categorical if and only if it has computable Scott rank. Turetsky \cite{Turetsky} showed:

\begin{theorem}[Turetsky \cite{Turetsky}]
	There is a computably categorical structure which has high Scott complexity and thus is not relatively $\Delta^1_1$-categorical.
\end{theorem}

\noindent Turetsky was also able to adapt this construction to build a structure of computable dimension 2:

\begin{theorem}[Turetsky \cite{Turetsky}]
	There is a structure of computable dimension 2 with two computable copies which are not hyperarithmetically isomorphic. This structure has high Scott complexity.
\end{theorem}

\noindent This is of particular interest because up until Turetsky's construction the only existing constructions \cite{Goncharov,CholakGoncharovKhoussainov,HirschfeldtKhoussainovShore} of a structure of computable dimension 2, using a method called the special component technique, build two computable copies which are not computably isomorphic, but which are $\Delta^0_3$ isomorphic.

\section{Scott Complexity for Certain Structures}

Scott complexity gives a good measuring stick for studying particular structures that we are interested in. Scott complexity measures the difficulty of characterizing a structure up to isomorphism, as well as (using Theorem \ref{thm:montalban}) the difficulty of understanding the automorphisms of the structure, and computing isomorphisms between different copies of the structure. By finding the Scott complexity of a particular structure, we come to understand the simplest way of characterizing it. Moreover, the problem of computing the Scott complexity of a structure gives us a way of testing our understanding that structure; indeed, one of the longstanding open questions is:
\begin{question}
	What is the Scott complexity of the pure transcendental field $\mathbb{Q}(x_1,x_2,\ldots)$?
\end{question}
\noindent This highlights the fact that we do not have a good understanding of \textit{pure transcendence bases}, i.e.\ transcendence bases which also generate the field extension. In contrast, Nielson transformations give us an excellent understanding of bases for free groups, allowing us to compute their Scott complexity; see Theorem \ref{thm:free-gp} below.

\medskip

To compute the Scott complexity of a structure, one usually writes down a Scott sentence of a particular complexity, and then one proves a lower bound showing that that complexity is optimal. So, for example, one might write down a $\Sigma_4$ Scott sentence for a structure $\mc{A}$, and then give a Wadge reduction $f$ from a $\bfSigma^0_4$-complete set $S$ to $\Iso(\mc{A})$, i.e.\ a continuous operator $f$ such that if $x \in S$ then $f(X) \cong \mc{A}$, and if $x \notin S$ then $f(x) \ncong \mc{A}$.

Much of the literature takes the approach of using index sets rather than Wadge reductions.

\begin{definition}
Fix a listing $\mc{M}_i$ of the partial computable structures in the language of $\mc{A}$. Then
\[ \mc{I}_{\mc{A}} = \{ i \in \omega : \mc{M}_i \cong \mc{A} \} \]
is the index set of $\mc{A}$.
\end{definition}

\noindent If $\mc{A}$ has a Scott sentence which is a computable $\Sigma_\alpha$ sentence then $\mc{I}_{\mc{A}}$ will be of complexity $\Sigma^0_\alpha$ as a subset of $\omega$. The same is true for $\Pi_\alpha$ and $\Pi^0_\alpha$, and $\mathrm{d-}\Sigma_\alpha$ and $\mathrm{d-}\Sigma^0_\alpha$ as well. So if we can prove, for example, that $\Iso(\mc{A})$ is $\Sigma^0_\alpha$ $m$-complete, then $\mc{A}$ cannot have a computable $\Pi_\alpha$ Scott sentence.

This strategy does not always work. Knight and McCoy \cite{KnightMcCoy14} showed that there is a structure (a particular subgroup of $\mathbb{Q}$) for which the index set is $m$-complete $\mathrm{d-}\Sigma^0_\alpha$, but there is no computable $\mathrm{d-}\Sigma_\alpha$ Scott sentence. The structure does however have an $X$-computable $\mathrm{d-}\Sigma_\alpha$ Scott sentence, for a particular low c.e.\ set $X$ used to build the subgroup. So at the computable level, the index set complexity and the Scott complexity do not always match up. To fix this, we must relativize.

\begin{definition}
	Fix a listing $\mc{M}^X_i$ of the partial $X$-computable structures in the language of $\mc{A}$. Then
	\[ \mc{I}^X_{\mc{A}} = \{ i \in \omega : \mc{M}^X_i \cong \mc{A} \} \]
	is the index set of $\mc{A}$ relative to $X$.
\end{definition}

\noindent So for example to show that $\mc{A}$ does not have any Scott sentence which is $\Pi_\alpha$ or simpler, we must show that there is a set $X$ such that for every $Y \geq_T X$, the index set $\mc{I}^Y_{\mc{A}}$ relative to $Y$ is $\Sigma^Y_\alpha$ $m$-complete. This essentially works out to the same things as giving a Wadge reduction.

\medskip

\begin{theorem}[Calvert, Harizanov, Knight, and Miller \cite{CalvertHarizanovKnightMiller06}; Calvert                                                                                                                                                                                                                                                     \cite{Calvert04}]\label{thm:vector}
	Let $\mc{A}$ be a $\mathbb{Q}$-vector space. Then:
	\begin{enumerate}
		\item If $\dim(\mc{A}) = 1$, then $\mc{A}$ has Scott complexity $\Pi_2$.
		\item If $\infty > \dim(\mc{A}) > 1$, then $\mc{A}$ has Scott complexity $\mathrm{d-}\Sigma_2$.
		\item If $\dim(\mc{A}) = \infty$, then $\mc{A}$ has Scott complexity $\Pi_3$.
	\end{enumerate}
\end{theorem}

\begin{theorem}[Calvert \cite{Calvert04}]
	Let $\mc{A}$ be an class of algebraically closed field. Then:
	\begin{enumerate}
		\item If $\mc{A}$ has finite transcendence degree, then $\mc{A}$ has Scott complexity $\mathrm{d-}\Sigma_2$.
		\item If $\mc{A}$ has infinite transcendence degree, then $\mc{A}$ has Scott complexity $\Pi_3$.
	\end{enumerate}
\end{theorem}

\noindent Both of these classes are examples of strongly minimal theories and hence admit a pregeometry of algebraic closure. Csima, Harrison-Trainor, and Mohammad (unpublished) 
have shown that the hardness portion of the results above are consequences of the existence of this pregeometry. Other classes of structures also admit a pregeometry with the right properties; another example is the class of real closed ordered fields. Calvert, Harizanov, Knight, and Miller \cite{CalvertHarizanovKnightMiller06} first computed the complexity of Archimedean real closed ordered fields.

\begin{theorem}[Calvert, Harizanov, Knight, and Miller \cite{CalvertHarizanovKnightMiller06}]
	Let $\mc{A}$ be an Archimedean real closed ordered field.
	\begin{enumerate}
	 	\item If the transcendence degree is $0$, then $\mc{A}$ has Scott complexity $\Pi_2$.
	 	\item If the transcendence degree of $\mc{A}$ is finite but not $0$, then $\mc{A}$ has Scott complexity $\mathrm{d-}\Sigma_2$.
	 	\item If the transcendence degree of $\mc{A}$ is infinite, then $\mc{A}$ has Scott complexity $\Pi_3$.
	\end{enumerate}
\end{theorem}

\noindent They also prove a number of results on reduced abelian $p$-groups.

\medskip

Groups have been a very fruitful area for proving interesting results about Scott complexity. This began with free groups, where Carson, Harizanov, Knight, Lange, McCoy, Morozov, Quinn, Safranski, and Wallbaum \cite{CHKLMMMQW} and McCoy and Wallbaum \cite{McCoyWallbaum} used Nielson transformations to compute the Scott complexities of free groups on various generators.

\begin{theorem}[\cite{CHKLMMMQW}, \cite{McCoyWallbaum}]\label{thm:free-gp}{\ }
	\begin{enumerate}
		\item For $n \geq 1$, the free group $F_n$ on $n$ generators has Scott complexity $\mathrm{d-}\Sigma_2$.
		\item The free group $F_\infty$ on infinitely many generators has Scott complexity $\Pi_4$.
	\end{enumerate}
\end{theorem}

Knight and Saraph \cite{KnightSaraph} continued the study of the Scott complexity of groups by looking at torsion-free abelian groups of finite rank and finitely generated groups. For groups of finite rank, Knight and Saraph show that there is always a $\Sigma_3$ Scott sentence.

\begin{theorem}[Knight and Saraph \cite{KnightSaraph}]
	Let $G$ be a torsion-free abelian group of finite rank. Then $G$ has a $\Sigma_3$ Scott sentence.
\end{theorem}

\noindent Moreover, they show that there are computable torsion-free abelian groups of rank one which have Scott complexity $\Sigma_3$.

\begin{theorem}[Knight and Saraph \cite{KnightSaraph}]
	If $G$ is a subgroup of $\mathbb{Q}$ in which $1$ is not infinitely divisible by any prime, but in which $1$ is finitely divisible by infinitely many primes, then $G$ has Scott complexity $\Sigma_3$.
\end{theorem}
\noindent One can also analyse the computability of these Scott sentences; see \cite{KnightSaraph} as well as \cite{Ho}.

Knight and Saraph also show that every finitely generated group has a $\Sigma_3$ Scott sentence, and indeed the same proof works for structures of any language.

\begin{theorem}[Knight and Saraph \cite{KnightSaraph}]\label{thm:fg-upper-bound}
	Every finitely generated structure has a $\Sigma_3$ Scott sentence.
\end{theorem}

\noindent By Theorem \ref{thm:free-gp} above, every finitely generated free group has a $\mathrm{d-}\Sigma_2$ Scott sentence. Moreover, the same is true of any finitely generated abelian group and of the dihedral group.

\begin{theorem}[Knight and Saraph \cite{KnightSaraph}]
	Every finitely generated abelian group has a $\mathrm{d-}\Sigma_2$ Scott sentence.
\end{theorem}

\begin{theorem}[Knight and Saraph \cite{KnightSaraph}]
	The infinite dihedral group has Scott complexity $\mathrm{d-}\Sigma_2$.
\end{theorem}

\noindent Next, Ho \cite{Ho} showed that polycyclic groups (which include nilpotent groups and abelian groups) and certain solvable groups also all have computable $\mathrm{d-}\Sigma^0_2$ Scott sentence.

\begin{theorem}[Ho \cite{Ho}]
	Every finitely generated polycyclic group has a $\mathrm{d-}\Sigma_2$ Scott sentence.
\end{theorem}

\noindent At the time, every finitely generated group that had been checked had a $\mathrm{d-}\Sigma_2$ Scott sentence, and so Knight and Saraph asked whether every finitely generated group has a $\mathrm{d-}\Sigma_2$ Scott sentence. Harrison-Trainor and Ho \cite{HTHo18} answered this question by showing that there is a finitely generated group with no $\mathrm{d-}\Sigma_2$ Scott sentence and hence with Scott complexity $\Sigma_3$.

\begin{theorem}[Harrison-Trainor and Ho \cite{HTHo18}]
	There is a finitely-generated computable group with Scott complexity $\Sigma_3$.
\end{theorem}

\noindent Thus the upper bound given by Theorem \ref{thm:fg-upper-bound} is best possible.

The proof of this last theorem relied on a characterizations using a number of different equivalent conditions on when a finitely generated structure has a $\mathrm{d-}\Sigma_2$ Scott sentence.
\begin{theorem}[A. Miller \cite{AMiller}, Harrison-Trainor and Ho \cite{HTHo18}, Alvir, Knight, and McCoy \cite{AlvirKnightMcCoy}]\label{thm:char}
	Let $\mc{A}$ be a finitely generated structure. The following are equivalent:
	\begin{enumerate}
		\item $\mc{A}$ has a $\mathrm{d-}\Sigma_2$ Scott sentence.
		\item $\mc{A}$ has a $\Pi_3$ Scott sentence.
		\item $\mc{A}$ is the only model of its $\Sigma_2$ theory.
		\item some generating tuple of $\mc{A}$ is defined by a $\Pi_1$ formula.
		\item every generating tuple of $\mc{A}$ is defined by a $\Pi_1$ formula.
		\item $\mc{A}$ does not contain a copy of itself as a proper $\Sigma_1$-elementary substructure.
	\end{enumerate}
\end{theorem}
\noindent (2) is due to A. Miller \cite{AMiller}, (3) is due to Harrison-Trainor and Ho \cite{HTHo21}, (4) and (5) are due to Alvir, Knight, and McCoy \cite{AlvirKnightMcCoy}, and (6) is due to Harrison-Trainor and Ho \cite{HTHo18}. It is condition (6) which seems to be the most useful for constructing examples, as it has a more semantic flavour. Alvir, Knight, and McCoy also have a nice characterization of when a structure has a \textit{computable} $\mathrm{d-}\Sigma_2$ Scott sentence.
\begin{theorem}[Alvir, Knight, and McCoy \cite{AlvirKnightMcCoy}]\label{thm:akmchar}
	Let $\mc{A}$ be a finitely generated structure. The following are equivalent:
	\begin{enumerate}
		\item $\mc{A}$ has a computable $\mathrm{d-}\Sigma_2$ Scott sentence.
		\item some generating tuple of $\mc{A}$ is defined by a computable $\Pi_1$ formula.
		\item every generating tuple of $\mc{A}$ is defined by a computable $\Pi_1$ formula.
	\end{enumerate}
\end{theorem}
\noindent There are indeed computable structures with a $\mathrm{d-}\Sigma_2$ Scott sentence, but no computable $\mathrm{d-}\Sigma_2$ Scott sentence; for example Theorem 2.8 of \cite{HTfp} gives an example of a computable finitely generated principal ideal domain which, as a module over itself, has a $\mathrm{d-}\Sigma_2$ Scott sentence but no computable $\mathrm{d-}\Sigma_2$ Scott sentence.

These general tools give us a way to study Scott complexity in other classes of finitely generated structures. The known result are shown in the table below. Those results which have not yet been cited are due to Harrison-Trainor and Ho \cite{HTHo18} and Harrison-Trainor \cite{HTfp}.

\medskip

\noindent \begin{tabular}{c|c}
	Some structure has no $\mathrm{d-}\Sigma_2$ Scott sentence & Every structure has a $\mathrm{d-}\Sigma_2$ Scott sentence \\\hline
	Rings & Commutative rings \\
	Modules & Modules over Noetherian rings \\
	& Fields \\
	& Vector spaces \\
	Groups & Abelian groups \\
	& Free groups \\
	& Torsion-free hyperbolic groups
\end{tabular}

\medskip

\noindent One important class of structures for which we do not know the answer is finitely presented groups.

\begin{question}[Question 1.7 of \cite{HTHo18} and Question 1 of \cite{AlvirKnightMcCoy}]
	Does every finitely presented group have a $\mathrm{d-}\Sigma_2$ Scott sentence?
\end{question}

\noindent The best progress towards answering this question is in \cite{HTfp}. There, Harrison-Trainor proves the following purely group-theoretic characterization of the finitely presented groups with a $\mathrm{d-}\Sigma_2$ Scott sentence:

\begin{theorem}[Harrison-Trainor \cite{HTfp}]\label{thm:fin-pres-group}
	Let $G$ be a finitely presented group. Then the following are equivalent:
	\begin{enumerate}
		\item $G$ does not have a $\mathrm{d-}\Sigma_2$ Scott sentence,
		\item $G$ contains a proper subgroup $H \cong G$ with the property that for every finite set of non-identity elements $a_1,\ldots,a_n \in G$, there is a normal subgroup $H' \subseteq G$ such that $G = H' \rtimes H$ and $a_1,\ldots,a_k \notin H'$.
	\end{enumerate} 
\end{theorem}

\noindent One can immediately see from this that any finitely presented Hopfian group has a $\mathrm{d-}\Sigma_2$ Scott sentence. (A Hopfian group is one which is not isomorphic to any of its proper quotients.) This includes cases such as finitely generated abelian groups and free groups for which we already knew this, as well as a few new examples such as finitely presented residually finite groups and torsion-free hyperbolic groups. We can also show that the Baumslag-Solitar groups $B(1,n)$, which were originally constructed specifically as examples of non-Hopfian finitely generated groups, have $\mathrm{d-}\Sigma_2$ Scott sentences.

Paolini \cite{Paolini} gave a sufficient condition, a weakening of Hopfianity, for an arbitrary finitely presented structure to have a $\mathrm{d-}\Sigma_2$ Scott sentence. We say that $\mc{A}$ is quasi-Hopfian if there is a finite generating set $\bar{a}$ of $\mc{A}$ such that whenever $f\colon \mc{A} \to \mc{A}$ is a surjective homomorphism of $\mc{A}$ which is injective on $\bar{a}$, $f$ is injective.

\begin{theorem}[Paolini \cite{Paolini}]
	Let $\mc{A}$ be a finitely presented structure. If $\mc{A}$ is \textit{quasi-Hopfian}, then $\mc{A}$ has a $\mathrm{d-}\Sigma_2$ Scott sentence.
\end{theorem}

\noindent Paolini also gives a condition on the automorphism group of a quasi-Hopfian computable structure which is sufficient to have a computable $\mathrm{d-}\Sigma_2$ Scott sentence. Paolini applies this to show that the free projective plane of rank 4, which is quasi-Hopfian but not Hopfian, has a computable $\textrm{d-}\Sigma_2$ Scott sentence. He also considers Coxeter groups of finite rank; these are Hopfian, and hence have $\mathrm{d-}\Sigma_2$ Scott sentences, but it takes more work to show that these Scott sentences are computable.

\medskip

The complexity of saying that a structure is finitely generated is $\Sigma_3$. When a finitely generated structure has Scott complexity $\Sigma_3$, is this just because it is $\Sigma_3$ to say that it is finitely generated? Or is the complexity of picking out the structure from among other finitely generated structures actually $\Sigma_3$? We can ask the analogous question whenever the complexity of a structure in a class $\mc{K}$ is the same or less than the complexity of $\mc{K}$. For such cases, we would like to separate the complexity of describing the class of structure from the complexity of describing the structure inside of the class.

\begin{definition}
	Let $\mc{A}$ be a structure in a class $\mc{K}$. We say that $\varphi$ is a \textit{quasi Scott sentence for $\mc{A}$ within $\mc{K}$} if, for countable $\mc{B} \in \mc{K}$,
	\[ \mc{B} \models \varphi \Longleftrightarrow \mc{B} \cong \mc{A}.\]
	The \textit{Scott complexity of $\mc{A}$ within $\mc{K}$} is $\Gamma \in \{\Sigma_\alpha,\Pi_\alpha,\mathrm{d-}\Sigma_\alpha\}$ if $\Gamma$ is the least complexity of a quasi Scott sentence for $\mc{A}$ within $\mc{K}$.
\end{definition}

\noindent To prove an upper bound, we just need to exhibit a quasi Scott sentence. We again prove lower bounds using a reduction. For example, to show that $\mc{A}$ has Scott complexity at least $\Sigma_\alpha$ within $\mc{K}$, we give a Wadge reduction $f$ from a $\bfSigma^0_\alpha$-complete set $S$ to $\Iso(\mc{A})$ with the additional property that $f(x) \in \mc{K}$ for all $x$.

For finitely generated structures, we have the following interesting result:

\begin{theorem}[Harrison-Trainor and Ho \cite{HTHo21}]
	Every finitely generated structure has a $\Pi_3$ quasi-Scott sentence within the class of finitely generated structures.
\end{theorem}

\noindent Recall that this means that for each finitely generated structure $\mc{A}$, there is a $\Pi_3$ sentence $\varphi$ such that $\mc{A}$ is the only finitely generated model of $\varphi$. In \cite{HTHo21}, Harrison-Trainor and Ho prove that there are finitely generated groups with no $\mathrm{d-}\Sigma_2$ quasi Scott sentence within the class of finitely generated structures. Thus the analogue for quasi Scott sentences of Theorem \ref{thm:akm} is not true: if a structure has a $\Sigma_3$ and a $\Pi_3$ quasi Scott sentence within the class of finitely generated structure, it is not necessarily true that it has a $\mathrm{d-}\Sigma_2$ quasi Scott sentence. There are two questions that are still open about quasi Scott sentences:

\begin{question}[Question 5.7 of \cite{HTHo18}]
	Give a characterization of the finitely generated structures which have a $\mathrm{d-}\Sigma_2$ quasi Scott sentence within the class of finitely generated structures.
\end{question}

\begin{question}[See Section 5.3 of \cite{HTHo21}]
	Is there a finitely generated structure with a $\mathrm{d-}\Sigma_2$ quasi Scott sentence within the class of finitely generated structures, but no $\mathrm{d-}\Sigma_2$ Scott sentence?
\end{question}

Recall from Theorem \ref{thm:free-gp} that the free group $F_\infty$ on infinitely many generators has Scott complexity $\Pi_4$. McCoy and Walbaum \cite{McCoyWallbaum} show that the class of free groups is axiomatizable by a $\Pi_4$ sentence. Within this class, $F_\infty$ is simpler to describe. ($F_1$ and $F_2$ also become simpler to describe within this class.)

\begin{theorem}[\cite{CHKLMMMQW}, \cite{McCoyWallbaum}]\label{thm:free-gp2}{\ }\\
	Let $\mc{FG}$ be the class of free groups.
	\begin{enumerate}
		\item The free group $F_1 \cong \mathbb{Z}$ on one generator has Scott complexity $\Pi_1$ within $\mc{FG}$.
		\item The free group $F_2$ on two generators has Scott complexity $\Pi_2$ within $\mc{FG}$.
		\item For $n > 2$, the free group $F_n$ on $n$ generators has Scott complexity $\mathrm{d-}\Sigma^0_2$ within $\mc{FG}$.
		\item The free group $F_\infty$ on infinitely many generators has Scott complexity $\Pi_3$ within $\mc{FG}$.
	\end{enumerate}
\end{theorem}

Quasi Scott sentences are also useful for talking about very simple structures, such as the 1-dimensional $\mathbb{Q}$-vector space; then Theorem \ref{thm:vector} below says that the Scott complexity of $\mc{V}$ is $\Pi_2$. However even just saying that $\mc{V}$ is a vector space is already $\Pi_2$; perhaps all the complexity is in saying that $\mc{V}$ is a vector space, and it is simpler than that to say that a vector space is 1-dimensional? However Calvert, Harizanov, Knight, and Miller \cite{CalvertHarizanovKnightMiller06} show that the 1-dimensional $\mathbb{Q}$-vector space has Scott complexity $\Pi_2$ within the class of $\mathbb{Q}$-vector spaces.

Quasi Scott sentences within the class of finitely generated structures are similar to \textit{quasi-finite axiomatizability}. A structure is quasi-finitely axiomatizable if there is a sentence $\varphi$ of finitary first-order logic such that the structure is the only finitely generated structure satisfying $\varphi$, that is, if it has a finitary quasi Scott sentence within the class of finitely generated structures. See the survey \cite{Nies}.

\section{Theories}

Given an $\mc{L}_{\omega_1 \omega}$ sentence $\varphi$ we can think of $\varphi$ as a theory defining the class $\{\mc{A} : \mc{A} \models \varphi\}$. This generalizes the notion of an elementary first-order theory, as given such a theory $T$, the conjunction $\bigwedge T$ of all the sentences in $T$ is an $\mc{L}_{\omega_1 \omega}$ sentence.

\subsection{Vaught's conjecture}

As previously mentioned, we give only the briefest overview of Vaught's conjecture and its connections with Scott sentences.

A central question in model theory is given a theory $T$, how many countable models can $T$ have? The continuum hypothesis would imply that if $T$ has infinitely many countable models, then it can only have countably many or continuum many countable models. In descriptive set theory, one often finds (e.g., for analytic sets) that even without the continuum hypothesis, one can prove that naturally occurring sets must have cardinality countable or continuum. Vaught's conjecture \cite{Vaught} asks whether this is true for the number of countable models of a theory.

\begin{conjecture}[Vaught]
	The number of countable models of an elementary first-order theory is either countable or continuum.
\end{conjecture}

\noindent It is natural to extend this to theories given by sentences of $\mc{L}_{\omega_1 \omega}$.

Morley was able to eliminate almost all of the other possibilities.

\begin{theorem}[Morley \cite{Morley70}]
	The number of countable models of an elementary first-order theory is either $\aleph_0$, $\aleph_1$, or continuum.
\end{theorem}

\noindent The standard proof of this theorem rests on what is called the Scott analysis together with Silver's theorem. The Scott analysis looks at the back-and-forth relations on all of the models of the theory, or, as in \cite{MarkerMTBook}, the analysis as in Definition \ref{def:scott-analysis}. This gives a level-by-level splitting of the models of the theory, e.g., at level $\alpha$ one splits the models into $\equiv_\alpha$-classes. Silver's theorem says that co-analytic equivalence relations have either countably many or continuum many equivalence classes \cite{Silver80}. So at each level $\alpha$ there are either countably many or continuum many equivalence classes. If the theory has fewer continuum-many models, at each level $\alpha$ there must only be countably many equivalence classes of models. Then there are $\omega_1$ levels, hence at most $\aleph_1$ countable models.

Of course Vaught's conjecture is automatically true if the continuum hypothesis is. Thus Vaught's conjecture is often stated in a way that is independent of the continuum hypothesis. We say a sentence $\varphi$ of $\mc{L}_{\omega_1 \omega}$ is \textit{scattered} if it does not have a perfect set of countable models (in the space $\Mod(\mc{L})$). Then one can state Vaught's conjecture as saying that any scattered theory has countably many countable models. Using the Scott analysis, one can show:

\begin{theorem}
	The following are equivalent:
	\begin{enumerate}
		\item $\varphi$ is scattered,
		\item for each countable $\alpha$, there are only countably many $\equiv_\alpha$-equivalence classes of models of $\varphi$,
		\item for each countable $\alpha$, there are only countably many models of $\varphi$ of Scott rank less than $\alpha$.
	\end{enumerate}
\end{theorem}

The Scott analysis, though not necessarily Scott rank or complexity, plays a central role in a large body of results about what a counterexample to Vaught's conjecture would have to look like. Often, these results are about the uncountable models of the theory. One example mentioning Scott rank explicitly is:


\begin{theorem}[Harrington]
	A counterexample to Vaught’s Conjecture has models	of arbitrarily large Scott rank below $\omega_2$.
\end{theorem}

\noindent In this theorem, we need the Scott rank for uncountable structures, e.g., by extending the Scott ranks based on back-and-forth relations. Harrington's proof was never published, but recently new proofs have independently been given by Baldwin, Friedman, Koerwien, and Laskowski \cite{BaldwinFriedmanKoerwienLaskowski}, Knight, Montalb\'an, and Schweber \cite{KnightMontalbanSchweber}, and Larson \cite{Larson}.

Montalb\'an showed that there are computability-theoretic ways of stating Vaught's conjecture; the proof also uses an analysis of the back-and-forth relations.

\begin{theorem}[ZFC+PD; Montalb\'an \cite{Montalban13}]
	Let $\varphi$ be an $\mc{L}_{\omega_1 \omega}$ sentence with uncountably many countable
	models. The following are equivalent:
	\begin{enumerate}
		\item[(V1)] $\varphi$ is a counterexample to Vaught’s conjecture.
		\item[(V2)] $\varphi$ satisfies hyperarithmetic-is-recursive on a cone, i.e., for all sufficiently powerful oracles $X$, any hyperarithmetic-in-$X$ model of $\varphi$ has an $X$-computable copy.
		\item[(V3)] There is an oracle relative to which,
		\[ \{ \text{Sp}(\mc{A}) \; : \; \mc{A} \models \varphi \} = \{\{X \in 2^\omega \; : \; \omega_1^X \geq \alpha\} : \alpha \in \omega_1\}.\]
	\end{enumerate}
\end{theorem}

As an example of an open question relating Scott rank and Vaught's conjecture, we have:

\begin{question}[Montalb\'an, Berkeley Vaught's Conjecture Workshop]
A \textit{minimal theory $T$} has the property that for any $\mc{L}_{\omega_1 \omega}$-sentence $\varphi$, exactly one of
$T \wedge \varphi$ or $T \wedge \neg \varphi$ has models of arbitrarily high Scott rank. Does the existence of
a minimal theory imply the existence of a counterexample to Vaught’s conjecture?
\end{question}

\subsection{Scott spectra}

The results of this section were proved using the notion of Scott rank defined using symmetric back-and-forth relations (as in Definition \ref{def:bfsym}) except with arbitrary tuples, replacing (2) of that definition with:
\begin{enumerate}
	\item[(2)] For $\alpha > 0$, $\bar{a} \sim_\alpha \bar{b}$ if for each $\beta < \alpha$ and $\bar{d}$ there is $\bar{c}$ such that $\bar{a} \bar{c} \sim_\beta \bar{b} \bar{d}$, and for all $\bar{c}$ there is $\bar{d}$ such that $\bar{a} \bar{c} \sim_\beta \bar{b} \bar{d}$.
\end{enumerate}
It is likely that the same or similar results can be proved for any of the definitions of Scott rank.

Let $\varphi$ be a $\Pi_2$ sentence which we think of as a theory. Many natural theories are of this form, e.g.\ most algebraic classes of interest such as algebraically closed fields or torsion groups. Montalb\'an asked whether there must be a model of $\varphi$ of somewhat low Scott complexity, say Scott rank 2 or less. One might expect that it would be possible to construct a model of any $\Pi_2$ theory without introducing any complicated types, resulting in a structure of low Scott rank. This turns out not to be the case; surprisingly, there are $\Pi_2$ Scott sentences all of whose models are very complicated.

\begin{theorem}[Harrison-Trainor \cite{HT18}]\label{thm:pi-two}
	Fix $\alpha < \omega_1$. There is a $\Pi_2$ sentence $T$ whose models all have Scott rank $\alpha$.
\end{theorem}

More generally, we can consider the collection of Scott ranks of all of the models of $\varphi$.

\begin{definition}
	Let $\varphi$ be an $\mc{L}_{\omega_1 \omega}$ sentence. The \textit{Scott spectrum} of $\varphi$ is the collection \[\{ \SR(\mc{A}) : \text{$\mc{A}$ is countable and $\mc{A} \models \varphi$}\}\] of Scott ranks of countable models of $\varphi$.
\end{definition}

\noindent There are many natural questions about Scott spectra. For example, what kind of gaps can one have---for which $\alpha < \beta$ can a theory have models of Scott rank $\alpha$ and $\beta$, but no models of Scott rank in between? Is there any relation between the complexity of the theory and the possible Scott spectrum? The following result of Sacks says that the Scott spectrum of a class is either bounded, or the class contains models of high Scott rank.

\begin{theorem}[Sacks \cite{Sacks83}]
	Let $\varphi$ be an $\mc{L}_{\omega_1 \omega}$ sentence. Suppose that the Scott rank of $\mc{M}$ is $\leq \omega_1^{\mc{M},\varphi}$ for all countable models $\mc{M}$ of $\varphi$. Then there is $\alpha < \omega_1^\varphi$ such that the Scott rank of every model of $\varphi$ is at most $\alpha$.
\end{theorem}

Harrison-Trainor \cite{HT18} gave a complete descriptive set-theoretic classification of the possible Scott spectra. To give this characterization, we need a few definitions.

\begin{definition}
	Let $(L,\leq)$ be a linear order. The \textit{well-founded part} $\wf(L)$ of $L$ is the largest initial segment of $L$ which is well-founded. The \textit{well-founded collapse} of $L$, $\wfc(L)$, is the order type of $L$ after we collapse the non-well-founded part $L \setminus \wf(L)$ to a single element.
\end{definition}
\noindent We can identify $\alpha \in \wf(L)$ with the ordinal which is the order type of $\{\beta \in L : \beta < \alpha\}$. We can also identify $\wf(L)$ with its order type. If $L$ is well-founded, with order type $\alpha$, then $\wfc(L) = \wf(L) = \alpha$. If $L$ is not well-founded, $\wfc(L) = \wf(L)+1$. We extend these operations to collections of linear orders by applying to operation to each order in the collection.

Then we get the following characterization of the Scott spectra:

\begin{theorem}[Harrison-Trainor \cite{HT18}; ZFC + PD]\label{thm:main-result2}
	The Scott spectra of $\mc{L}_{\omega_1 \omega}$-sentences are exactly the sets of the form:
	\begin{enumerate}
		\item $\wf(\mc{C})$,
		\item $\wfc(\mc{C})$, or
		\item $\wf(\mc{C}) \cup \wfc(\mc{C})$
	\end{enumerate}
	where $\mc{C}$ is a $\bfSigma^1_1$ class of linear orders of $\omega$.
\end{theorem}

\noindent The more difficult and most useful part of this characterization is to show that each such collection of ordinals is the Scott spectrum of a theory. This direction is provable within ZFC and does not use the set-theoretic assumption of projective determinacy. One can use this direction to construct interesting examples of Scott spectra, such as the collection of successor ordinals, or the collection of admissible ordinals.

Harrison-Trainor also proved that $\Pi_2$ theories are enough to get all of the Scott spectra.

\begin{theorem}[Harrison-Trainor \cite{HT18}]\label{thm:all-pi-two}
	Each Scott spectrum is the Scott spectrum of a $\Pi^{\infi}_2$ sentence.
\end{theorem}

\noindent For example, for every pair of ordinals $\alpha < \beta$, $\{\alpha,\beta\}$ is a Scott spectrum.

\subsection{Scott height}

Of the countably many computable $\mc{L}_{\omega_1 \omega}$ sentences, some of them have a Scott spectrum which is unbounded below $\omega_1$, and others are bounded below $\omega_1$. Since there are only countably many computable sentences, of all those that are bounded below $\omega_1$, there must be some countable ordinal which is the least common upper bound for all of them. Using $\mc{L}^{\comp}_{\omega_1 \omega}$ for the computable $\mc{L}_{\omega_1 \omega}$ sentences, we define:

\begin{definition}
	The Scott height $\sh(\mc{L}^{\comp}_{\omega_1 \omega})$ of $\mc{L}^{\comp}_{\omega_1 \omega}$ is the least ordinal $\alpha < \omega_1$ such that if $\varphi$ is a computable $\mc{L}_{\omega_1 \omega}$ sentence whose Scott spectrum is bounded below $\omega_1$, then the Scott spectrum of $\varphi$ is bounded below $\alpha$.
\end{definition}

\noindent The Scott height is a way of measuring how far up the ordinals the computable sentences can ``reach''.

Sacks and Marker made significant progress towards computing the Scott height of $\mc{L}^{\comp}_{\omega_1 \omega}$. Sacks \cite{Sacks83} showed that $\sh(\mc{L}^{\comp}_{\omega_1,\omega}) \leq \delta^1_2$, where $\delta^1_2$ is the least ordinal which has no $\Delta^1_2$ presentation. Marker \cite{Marker90} was able to resolve this question for pseudo-elementary classes.

\begin{definition}
	A class $\mc{C}$ of structures in a language $\mc{L}$ is an $\mc{L}_{\omega_1 \omega}$-pseudo-elementary class ($PC_{\mc{L}_{\omega_1 \omega}}$-class) if there is an $\mc{L}_{\omega_1 \omega}$-sentence $\varphi$ in an expanded language $\mc{L}' \supseteq \mc{L}$ such that the structures in $\mc{C}$ are the reducts to $\mc{L}$ of the models of $\varphi$. $\mc{C}$ is a computable $PC_{\mc{L}_{\omega_1 \omega}}$-class if $\varphi$ is a computable sentence.
\end{definition}

\noindent We can define the Scott height of $PC_{\mc{L}_{\omega_1 \omega}}$ in a similar way to the Scott height of $\mc{L}^{\comp}_{\omega_1 \omega}$, except that now we consider all $\mc{L}_{\omega_1 \omega}$-pseudo-elementary classes which are the reducts of the models of a computable sentence. Marker \cite{Marker90} showed that $\sh(PC^{\comp}_{\mc{L}_{\omega_1 \omega}}) = \delta^1_2$. Finally, Harrison-Trainor \cite{HT18} showed that every Scott spectrum of a pseudo-elementary $\mc{L}_{\omega_1 \omega}$-class is the Scott spectrum of an $\mc{L}_{\omega_1 \omega}$-sentence.

\begin{theorem}[Harrison-Trainor \cite{HT18}]\label{thm:main-result-pseudo}
	Every Scott spectrum of a pseudo-elementary $\mc{L}_{\omega_1 \omega}$-class is the Scott spectrum of an $\mc{L}_{\omega_1 \omega}$-sentence. Moreover, every Scott spectrum of a computable pseudo-elementary $\mc{L}_{\omega_1 \omega}$-class is the Scott spectrum of a computable $\mc{L}_{\omega_1 \omega}$-sentence.
\end{theorem}

Combined, all of these results allow us to compute the Scott height.

\begin{corollary}[Sacks \cite{Sacks83}, Marker \cite{Marker90}, Harrison-Trainor \cite{HT18}]\label{thm:sh}
	$\sh(\mc{L}^{\comp}_{\omega_1,\omega}) = \delta^1_2$.
\end{corollary}

Marker also studies Scott heights for other types of classes. We say that a class is co-pseudo-elementary $\mc{L}_{\omega_1 \omega}$, or $c\mc{PC}_{\mc{L}_{\omega_1 \omega}}$, if it is the complement of a pseudo-elementary class. Then Marker \cite{Marker90} shows that:

\begin{theorem}[Marker \cite{Marker90}]
	We have:
	\[ \sh(c\mc{PC}_{\mc{L}_{\omega_1 \omega}}) = \gamma^1_2.\]
\end{theorem}

\noindent Here, $\gamma^1_2$ is the least ordinal such that any non-empty $\Pi^1_2$ class of ordinals contains an ordinal below $\gamma^1_2$. (By a $\Pi^1_2$ class of ordinals, we mean a $\Pi^1_2$ class of presentations of linear orders on domain $\omega$.) We have $\gamma^1_2 > \delta^1_2$.

\bibliography{References}
\bibliographystyle{alpha}

\end{document}